\documentclass[11pt]{amsart}
\usepackage[utf8]{inputenc}
\usepackage[a4paper,margin=3cm]{geometry}
\usepackage[pdftex,pdfborder={0 0 0},colorlinks,pagebackref]{hyperref} 
\usepackage{latexsym,lmodern,graphicx,enumitem,amssymb,dsfont}
\usepackage{bm}

\renewcommand\epsilon\varepsilon
\renewcommand\phi\varphi
\renewcommand\geq\geqslant
\renewcommand\leq\leqslant
\renewcommand\ln\log

\newcommand\RR{\mathbb{R}}

\newcommand\ab\allowbreak
\newcommand{\bone}{\mathbf{1}}
\newcommand{\ent}{\mathrm{Ent}}
\newcommand{\var}{\mathrm{Var}}

\newcommand{\tbar}{\overline{\mathcal{T}}}
\newcommand{\E}{\mathbb{E}}
\renewcommand{\P}{\mathbb{P}}

\theoremstyle{definition}
\newtheorem{Def}{Definition}[section]

\theoremstyle{plain}
\newtheorem{Pro}[Def]{Proposition}
\newtheorem{Lem}[Def]{Lemma}
\newtheorem{The}[Def]{Theorem}

\theoremstyle{remark}

\newtheorem{Rem}[Def]{Remark}

\makeatletter
\def\@MRExtract#1 #2!{#1}     
\renewcommand{\MR}[1]{
  \xdef\@MRSTRIP{\@MRExtract#1 !}%
  \href{http://www.ams.org/mathscinet-getitem?mr=\@MRSTRIP}{MR-\@MRSTRIP}}
\makeatother

\begin{document}
\title[Characterization of weak transport-entropy inequalities ]{Characterization  of a class of weak transport-entropy inequalities  on the line}
\author[N. Gozlan, C. Roberto, P.-M. Samson, Y. Shu, P. Tetali]{Nathael Gozlan, Cyril Roberto, Paul-Marie Samson, Yan Shu, Prasad Tetali}
\date{\today}

\thanks{ Supported by the grants ANR 2011 BS01 007 01,  ANR 10 LABX-58, ANR11-LBX-0023-01; the last author is supported by the NSF grants DMS 1101447 and 1407657, and is also grateful for the hospitality of  Universit\'e Paris Est Marne La Vall\'ee. The authors acknowledge the kind support of the American Institute of Mathematics (AIM)}

\address{Universit\'e Paris Est Marne la Vall\'ee - Laboratoire d'Analyse et de Math\'e\-matiques Appliqu\'ees (UMR CNRS 8050), 5 bd Descartes, 77454 Marne la Vall\'ee Cedex 2, France}
\address{Universit\'e Paris Ouest Nanterre La D\'efense - Modal'X, 200 avenue de la R\'epublique 92000 Nanterre, France}
\address{School of Mathematics \& School of Computer Science, Georgia Institute of Technology,
Atlanta, GA 30332}
\email{natael.gozlan@u-pem.fr, c.roberto@math.cnrs.fr,  paul-marie.samson@u-pem.fr, yan.shu@u-paris10.fr, tetali@math.gatech.edu}
\keywords{Transport inequalities, concentration of measure, majorization}
\subjclass{60E15, 32F32 and 26D10}

\maketitle
\begin{abstract} We  study an optimal weak transport cost related to the notion of convex order between probability measures. On the real line, we show that this weak transport cost is reached for a coupling that does not depend on the underlying cost function. As an application, we give a  necessary and sufficient condition for weak transport-entropy inequalities  (related to concentration of convex/concave functions) to hold on the line. In particular, we obtain a weak transport-entropy form of the convex Poincaré inequality in dimension one.    
\end{abstract}
\section{Introduction} 
The aim of this paper is to study a weak transport cost and its associated weak transport-entropy inequality, both introduced by  four of the  authors in \cite{GRST15}, on the real line. In order to present our results, we shall first introduce the various mathematical objects of interest to us, placing and motivating their significance within the classical theory of optimal transport and its connection with the concentration of measure phenomenon.

\medskip

Throughout the paper,  $\mathcal{P}(\RR)$ denotes the set of Borel probability measures on $\RR$ and  
$\mathcal{P}_1(\RR):=\left\{ \mu \in \mathcal{P}(\RR) : \int_\mathbb{R} |x| \mu(dx) < \infty\right\}$, the subset of  probability measures  having a finite first moment.

Let $\theta:\RR^+ \to \RR^+$, with $\theta(0)=0$, be a measurable function referred to as the \emph{cost} function. Then, the usual optimal \emph{transport cost}, in the sense of Kantorovich, between two probability measures $\mu$ and $\nu$ on $\RR$ is defined by
\begin{equation}\label{transport}
\mathcal{T}_\theta(\nu,\mu):=\inf_{\pi}\iint \theta(|x-y|)\,\pi(dxdy),
\end{equation}
where the infimum runs over the set of couplings $\pi$ between $\mu$ and $\nu$, \textit{i.e.,}\ probability measures on $\RR^2$ such that $\pi(dx\times \RR)=\mu(dx)$ and $\pi(\RR\times dy)=\nu(dy)$.

Since the works by Marton \cite{Mar86, Mar96b,Mar96a} and Talagrand \cite{Tal96}, these transport costs have been extensively used as a tool to reach concentration properties for measures on product spaces. More precisely, optimal transport is related to the concentration of measure phenomenon via   the so-called transport-entropy inequalities that we now recall. A probability measure $\mu$ on $\RR$ is said to satisfy the transport-entropy inequality $\mathrm{T}(\theta)$, if for all $\nu \in \mathcal{P}(\RR)$, it holds
\begin{equation}\label{eq:Ttheta-intro}
\mathcal{T}_\theta(\nu,\mu) \leq \mathrm{H}(\nu|\mu), 
\end{equation}
where $\mathrm{H}(\nu|\mu)$ denotes the relative entropy (also called Kullback-Leibler distance) of $\nu$ with respect to $\mu$, defined by
\[
\mathrm{H}(\nu|\mu):=\int \log\left(\frac{d\nu}{d\mu}\right)\,d\nu,
\]
if $\nu$ is absolutely continuous with respect to $\mu$, and $\mathrm{H}(\nu|\mu):=\infty$ otherwise. Note that 
we focus on the line, but that all the above definitions easily generalize to probability measures on a general metric space. As a special case, as proved by Talagrand in his seminal paper \cite{Tal96}, Inequality \eqref{eq:Ttheta-intro} is satisfied by the standard Gaussian measure for the cost $\theta(x)=x^2/2$. By extension, we shall say that $\mu$ satisfies the inequality $\mathrm{T}_2(C)$ (often referred to as ``Talagrand's inequality'' in the literature), if \eqref{eq:Ttheta-intro} holds for a cost function of the form $\theta(x)=x^2/C$, for some $C>0$.  We refer to the books or survey \cite{Led01,SurveyGL,Vil09,BLM13} for a complete presentation of transport-entropy inequalities and of the concentration of measure phenomenon, as well as for bibliographic references in the field.

In the next few lines we shall shortly discuss the consequences of transport-entropy inequalities in terms of concentration. For simplicity we may only consider Inequality $\mathrm{T}_2(C)$. 

As discovered by Marton and Talagrand, when a probability $\mu$ satisfies $\mathrm{T}_2(C)$, then for \emph{all} positive integers $n$, and all functions $f:\RR^n \to \RR$ which are $1$-Lipschitz with respect to the Euclidean norm on $\RR^n$, it holds
\begin{equation}\label{eq:dim-free}
\mu^n(f > \mathrm{med}(f) + t) \leq e^{-(t-t_o)^2/C},\qquad \forall t \geq t_o := \sqrt{C\log (2)},
\end{equation}
where $\mathrm{med}(f)$ denotes the median of $f$ under $\mu^n.$ We refer to \cite{Led01, BLM13} for a presentation of the numerous applications of this type of dimension-free concentration of measure inequalities. Conversely, it was shown by the first named author in \cite{Goz09} that a probability $\mu$ satisfying the dimension-free Gaussian concentration \eqref{eq:dim-free} necessarily satisfies $\mathrm{T}_2(C)$, thus giving to Inequality $\mathrm{T}_2$ a special status among other functional inequalities appearing in the concentration of measure literature.
The key argument explaining why Talagrand's inequality implies the dimension-free concentration behavior \eqref{eq:dim-free} is the well-known tensorisation property enjoyed in general by inequalities of the form $\mathrm{T}(\theta)$ (see \textit{e.g.}\  \cite{SurveyGL}). The tensorisation property shows in particular that if $\mu$ satisfies $\mathrm{T}_2(C)$, then the $n$-fold product measure $\mu\otimes\cdots\otimes \mu$ also satisfies $\mathrm{T}_2(C)$ (on $\RR^n$) with the same constant $C$.

More generally, given a measure on a product space (which is not necessarily a product measure), and assuming that each of its conditional one-dimensional marginals satisfies a transport-entropy inequality, several authors have obtained, using different non-independent tensorisation strategies, transport-entropy inequalities for the whole measure under weak dependence assumptions (see for instance \cite{DGW04,Wu06,Mar13,Wan14}). Then, the transport-entropy inequality for the whole measure leads again to concentration properties using the same classical arguments as in the product case. 
Thus, in many situations one is reduced to verify the one-dimensional transport-entropy inequalities, 
and therefore it is of a real interest to characterize those probability measures $\mu$ on $\RR$ that satisfy the inequality  
$\mathrm{T}_2$ and more generally $\mathrm{T}(\theta)$ for a general cost function $\theta$. 

In this direction, the first-named author obtained, in \cite{Goz12}, necessary and sufficient conditions for $\mathrm{T}(\theta)$ to hold, when the cost function $\theta:\RR^+\to\RR^+$ is continuous, convex and quadratic near $0$. In order to present such conditions, we need to introduce some notations.
Denote by $F_\mu(x):=\mu(-\infty,x]$, $x \in \RR$, the cumulative distribution function of a probability measure $\mu$ 
and by $F_\mu^{-1}$ its general inverse defined by
\[
F_\mu^{-1}(u):=\inf\{x\in\RR, F_\mu(x)\geq u\}\in \RR \cup\{\pm\infty\},\qquad \forall u\in [0,1].
\]
The conditions obtained in \cite{Goz12} are expressed in terms of the behavior of the modulus of continuity of the non-decreasing map $U_\mu:=  F_\mu^{-1}\circ F_\tau$, where $\tau$ is the symmetric exponential distribution on $\RR$,
\[
\tau(dx):=\frac{1}{2}e^{-|x|}\,dx\,,
\]
so that
\[
U_\mu(x)=\left\{\begin{array}{ll}
F_\mu^{-1}\left(1-\frac12 e^{-|x|}\right) &\mbox{if } x\geq 0\\
F_\mu^{-1}\left( e^{-|x|}\right) &\mbox{if } x\leq 0 .
\end{array}\right.
\]
By construction, $U_\mu$ is the unique left-continuous and non-decreasing map transporting $\tau$ onto $\mu$ (\textit{i.e.}\ $\int f \circ U_\mu d\tau = \int f d\mu$ for all $f$). 
In the special case of the inequality $\mathrm{T}_2$, the characterization of \cite{Goz12} reads as follows: a probability measure $\mu$ satisfies $\mathrm{T}_2(C)$ for some $C$ if and only the following holds
\begin{itemize}
\item for some  constant $b>0$ and all $u \geq 0$, it holds
\[
\sup_{x \in \RR} (U_\mu(x+u)-U_\mu(x)) \leq \frac{1}{b} \sqrt{1+u},
\]
\item for some constant $c>0$ and all $f$ of class  $\mathcal{C}^1$, the Poincar\'e inequality holds 
\begin{equation}\label{eq:Poinc-intro}
\mathrm{Var}_\mu(f) \leq c \int {f'}^2\,d\mu\,.
\end{equation}
\end{itemize}
We refer to \cite{Goz12} for a precise quantitative relation between $C,b$ and $c$.

\medskip

In the present paper, partly following \cite{Goz12}, we focus on the study of a new weak transport-entropy inequality introduced in \cite{GRST15} that is related to a weak type of dimension-free concentration.
More precisely, in dimension one, we consider the weak optimal transport cost of $\nu$ with respect to $\mu$  defined by
\[
\overline{ \mathcal{T}}_\theta(\nu|\mu)=\inf_{\pi}\int \theta\left(\left|x-\int y\,p(x,dy)\right|\right)\,\mu(dx)\,,
\]
where the infimum runs over all couplings $\pi(dxdy)=p(x,dy)\mu(dx)$ of $\mu$ and $\nu$, and where $p(x,\,\cdot\,)$ denotes the disintegration kernel of $\pi$ with respect to its first marginal. The notation bar comes from the barycenter entering in its definition.
Note that, contrary to the usual transport cost, $\overline{ \mathcal{T}}_\theta$ is not symmetric. Also, in terms of random variables, one has the following interpretation $\overline{ \mathcal{T}}_\theta(\nu|\mu)=\inf \E \left(\theta(|X-\E(Y|X)|)\right)$
whereas $
\mathcal{T}_\theta(\nu,\mu)=\inf \E \left(\theta(|X-Y|)\right)$,
where in both cases the infimum runs over all random variables $X,Y$ such that $X$ follows the law $\mu$ and $Y$ the law $\nu$.
As a consequence, when $\theta$ is convex,  by Jensen's inequality, one has 
$\overline{ \mathcal{T}}_\theta(\nu|\mu)\leq \mathcal{T}_\theta(\nu,\mu)$.
Therefore, if a measure $\mu$ satisfies   $\mathrm{T}(\theta)$ then it also satisfies  the following weaker transport-entropy inequalities.
\begin{Def}\label{defitrans}
Let $\theta: \RR^+ \to \RR^+$ be a convex cost function. A probability measure $\mu$ on $\RR$ is said to satisfy the transport-entropy inequality
\begin{itemize}[leftmargin=.8in]
\item[$\overline{\mathrm{T}}^+(\theta)$:] if for all $\nu\in\mathcal{P}_1(\RR)$, it holds
\[
\overline{\mathcal{T}}_\theta(\nu|\mu)\leq \mathrm{H}(\nu|\mu);
\]
\item[$\overline{\mathrm{T}}^-(\theta)$:] if for all $\nu\in\mathcal{P}_1(\RR)$, it holds
\[
\overline{\mathcal{T}}_\theta(\mu|\nu)\leq \mathrm{H}(\nu|\mu);
\]
\item[$\overline{\mathrm{T}}(\theta)$:] if $\mu$ satisfies both $\overline{\mathrm{T}}^+(\theta)$ and $\overline{\mathrm{T}}^-(\theta)$.
\end{itemize}
\end{Def}

In Section \ref{dualc}, we recall a dual formulation of these weak transport inequalities in terms of infimum convolution operators. In particular, the inequality $\overline{\mathrm{T}}(\theta)$ appears as the dual formulation of the so-called convex $(\tau)$-property introduced by Maurey \cite{Mau91} (see also \cite{Sam03} for further development).   

The above defined weak transport-entropy inequalities are of particular interest since the class of measures satisfying such inequalities also includes discrete measures on $\RR$ such as Bernoulli, binomial and Poisson measures \cite{GRST15, Sam03}. In comparison, the classical Talagrand's transport inequality (and more generally any $\mathrm{T}(\theta)$) is never satisfied by a discrete probability measure unless it is a Dirac\footnote{Indeed, as mentioned above, the  Poincaré inequality is a consequence of Talagrand's transport inequality that forces the support of $\mu$ to be connected.}. Moreover these weak transport-entropy inequalities also enjoy the tensorisation property (see \cite[Theorem 4.11]{GRST15}), thus connecting them to a special dimension-free concentration behavior. For instance, as shown in \cite[Corollary 5.11]{GRST15}, a probability measure $\mu$ satisfies $\overline{\mathrm{T}}_2(C)$ (\textit{i.e.} $\overline{\mathrm{T}}(\theta)$ with $\theta(x)=x^2/C$, $x\in \RR$) if and only if, for all positive integers $n$ and all \emph{convex} or \emph{concave} functions $f \colon \RR^n \to \RR$ which are also $1$-Lipschitz for the Euclidean norm on $\RR^n$, it holds
\[
\mu^n(f >\mathrm{med}(f) + t) \leq e^{-(t-t_o)^2/C'},\qquad \forall t\geq t_o,
\]
where $t_o,C'>0$ are constants related to $C$ (see \cite{GRST15} for a precise and more general statement). 

With these motivations in mind we shall prove the following characterization (the main result of the paper) of the transport inequalities $\overline{\mathrm{T}}(\theta)$ associated to any convex cost function $\theta$ which is quadratic near $0$.
In the sequel we may use the following standard notation $\theta(a\,\cdot\,)$ for the function $x \mapsto \theta(ax)$.

\begin{The}\label{thefinal}
Let $\mu\in \mathcal{P}_1(\RR)$, $t_o>0$ and $\theta:\RR^+\rightarrow \RR^+$ be a convex cost function such that $\theta(t)=t^2$ for all $t\leq t_o$.  The following propositions are equivalent:
\begin{enumerate}
\item[$(i)$] There exists $a>0$ such that $\mu$ satisfies $\overline{\mathrm{T}} (\theta(a\,\cdot\,))$.
\item[$(ii)$] There exists $b>0$ such that for all $u>0$,
\[
\sup_x \left(U_\mu(x+u)-U_\mu(x)\right)\leq \frac{1}{b}\theta^{-1}(u+t_o^2).
\]
\end{enumerate}
Moreover, constants are related as follows: 
\begin{enumerate}
\item[] $(i)$ implies $(ii)$ with $b = a \kappa_1$
\item[]$(ii)$ implies $(i)$, with $a = b\kappa_2$,
\end{enumerate}
where $\kappa_1 := \frac{t_o}{8 \theta^{-1}(\log(3) +t_o^2)}$ and $\kappa_2 := \frac{\min(1,t_o)}{210 \theta^{-1}(2+t_o^2)}$. 
\end{The}

\begin{Rem}
In comparison with the characterization of the inequalities $\mathrm{T}(\theta)$ given in \cite{Goz12}, one sees that only the condition on the modulus of continuity of $U_\mu$ remains. Nevertheless, as we shall explain below, the Poincar\'e Inequality has not completely disappeared from the picture. 

Also, denoting by $\Delta_\mu$ the modulus of continuity of $U_\mu$ defined by
\[
\Delta_\mu(h)= \sup\left\{U_\mu(x+u)-U_\mu(x), x\in \RR, 0\leq u\leq h\right\},\quad h\geq 0,
\]
Condition $(ii)$ asserts that 
\[
\Delta_\mu(h)\leq \frac{1}{b}\theta^{-1}(h+t_o^2).
\]
Therefore $\Delta_\mu$ is bounded above near zero but does not necessarily go to zero, as $h$ goes to zero. In fact, if the measure $\mu$ is discrete and not a Dirac measure, the support of $\mu$ is not connected and so there exist $a<b$ with $a$ and $b$ in the support of $\mu$ such that $\mu(]a,b[)=0$. In that case, we may easily check that for all $h>0$, $b-a\leq \Delta_\mu(h)$.
This shows that in a discrete setting $\lim_{h\to 0}\Delta_\mu(h)>0$.
 \end{Rem}

The proof of Theorem \ref{thefinal}, given in Section \ref{sectionfinal}, is based on a refined study of the weak transport cost
$\overline{\mathcal{T}}_\theta(\mu|\nu)$ of independent interest. Indeed, we shall prove that, in dimension 1, all the optimal weak transport costs $\overline{\mathcal{T}}_\theta(\mu|\nu)$ are achieved by the same coupling independently of the convex cost function $\theta$. This result is well-known  for the classical transport cost $\mathcal{T}_\theta$. More precisely,
 it follows from the works by Hoeffding, Fr\'echet and Dall'Aglio \cite{DA56, Fr60, H40} (see also \cite{CSS76}) that
$\mathcal{T}_\theta(\mu,\nu)=\int \theta\left(|x-T_{\nu,\,\mu}(x)|\right)\,\nu(dx)$ where $T_{\nu,\,\mu}:=F_\mu^{-1}\circ F_\nu$.
In particular, given any two convex costs $\theta_1, \theta_2$, it holds 
$\mathcal{T}_{\theta_1+\theta_2}(\mu,\nu) = \mathcal{T}_{\theta_1}(\mu,\nu) + \mathcal{T}_{\theta_2}(\mu,\nu)$.
Our second main result reads as follows (recall that $\nu_1 \preceq \nu_2$ means that $\int f\,d\nu_1 \leq \int f\,d\nu_2$ for all convex functions $f\colon\RR \to \RR$ (one says that $\nu_1$ is dominated by $\nu_2$ in the convex order)).
\begin{The}\label{The:OT}
Let $\mu,\nu \in \mathcal{P}_1(\RR)$ ; there exists a probability measure $\hat{\gamma}$ dominated by $\nu$ in the convex order,  $\hat{\gamma} \preceq \nu$, such that  \emph{for all} convex cost functions $\theta$ it holds
\[
\overline{\mathcal{T}_\theta}(\nu|\mu) = \mathcal{T}_\theta(\hat{\gamma},\mu).
\]
In particular, for any two convex cost functions $\theta_1, \theta_2$, it holds 
\begin{equation}\label{eval}
\overline{\mathcal{T}}_{\theta_1+\theta_2}(\mu,\nu) = \overline{\mathcal{T}}_{\theta_1}(\mu,\nu) 
+ \overline{\mathcal{T}}_{\theta_2}(\mu,\nu).
\end{equation}
\end{The}

The notion of convex ordering, characterized by Strassen \cite{Str65} in terms of martingales, will turn out to be crucial in the understanding of the weak transport costs $\overline{\mathcal{T}_\theta}$. In Section \ref{ordering}, we recall certain classical properties of the convex order and in particular its geometrical meaning (in discrete setting) given by Rado's theorem \cite{Rad52} (see Theorem \ref{Rado}). From this geometrical interpretation, we shall obtain an intermediate outcome (Theorem \ref{goodprop}) that might be interpreted as a discrete version of Theorem \ref{The:OT}. Finally, the proof of Theorem \ref{The:OT}, given in Section \ref{propopti}, will follow by an approximation argument.

\medskip

With the result of Theorem \ref{The:OT} in hand, we can briefly introduce the main ideas of the proof of Theorem \ref{thefinal}. 
Following \cite{Goz12}, the weak transport-entropy inequality  $(i)$ will follow from $(ii)$ by decomposition of the weak optimal cost $\theta$ into two parts. One part is related  to the quadratic behavior of $\theta$ on $[0,t_o]$ and the  second part is related to its behavior for $t\geq t_o$: one has $\theta\leq \theta_1+\theta_2$ with 
\[
\theta_1(t):=t^2\bone_{[0,t_o]}(t)+(2tt_o-t_o^2)\bone_{[t_o,+\infty)}(t),
\]
and 
\[
\theta_2(t):=[\theta(t)-t^2]_+=(\theta(t)-t^2)\bone_{[t_o,+\infty)}(t), \qquad t\in \RR.
\]
Therefore, by Theorem \ref{The:OT},
\begin{equation*}
\overline{\mathcal{T}}_{\theta(a\,\cdot\,)} (\nu |\mu) 
\leq 
\overline{\mathcal{T}}_{(\theta_1+\theta_2)(a\,\cdot\,)} (\nu |\mu)=  \overline{\mathcal{T}}_{\theta_1(a\,\cdot\,)} (\nu |\mu) + \overline{\mathcal{T}}_{\theta_2(a\,\cdot\,)} (\nu |\mu) \leq 2  \mathrm{H}(\nu|\mu)\,,
\end{equation*}
for a proper choice of the constant $a$, where the last inequality will follow by relating the condition appearing in $(ii)$ 
to the two weak transport-entropy inequalities with cost $\theta_1(a\,\cdot\,)$ and $\theta_2(a\,\cdot\,)$.
More precisely, following \cite[Theorem 2.2]{Goz12}, we will show in Theorem \ref{thebeta} that $(ii)$ characterizes the 
weak transport-entropy inequality $\overline{\mathrm{T}}(\theta_2(a\,\cdot\,))$ while, as stated in the next Theorem (our last main result), the weaker condition $\sup_x \left(U_\mu(x+1)-U_\mu(x)\right)\leq h$,
for some $h>0$, characterizes the weak transport-entropy inequality $\overline{\mathrm{T}}(\theta_1(a\,\cdot\,))$ which thus appears (thank to \cite{BG99b}) to be also equivalent to the Poincaré inequality \eqref{eq:Poinc-intro} restricted to convex functions.

\begin{The}\label{thm:equiv-Poincare}
Let $\mu$  be a probability measure on $\RR$. The following assertions are equivalent:
\begin{enumerate}
\item[$(i)$] There exists $h>0$ such that
\[
  \sup_{x\in \RR}[U_\mu(x+1)-U_\mu(x)]\leq h.
\]
\item[$(ii)$] There exists $C>0$ such that for all convex function $f$ on $\RR$ it holds
\begin{equation*}
  \var_\mu(f)\leq C \int_\RR f'^2\,d\mu.
\end{equation*}

\item[$(iii)$] There exist $D,l_o>0$ such that $\mu$ satisfies $\overline{\mathrm{T}}^-(\alpha)$ and $\overline{\mathrm{T}}^+(\alpha)$. Namely, it holds 
\[
\overline{\mathcal{T}}_\alpha (\mu|\nu) \leq H(\nu|\mu)
\qquad \mbox{and} \qquad
\overline{\mathcal{T}}_\alpha (\nu|\mu) \leq H(\nu|\mu)\qquad \forall \nu \in \mathcal{P}_1(\RR),
\]
for the cost function $\alpha$ defined by
\[
\alpha(u) = \left\{\begin{array}{ll}\frac{u^2}{2D} & \text{ if } |u| \leq l_oD \\ l_o|u|-l_o^2D/2 & \text{ if } |u|>l_oD . \end{array}\right.
\]
\end{enumerate}
In particular, with $\kappa:=5480$ and $c:=1/(10 \sqrt{2})$,
\begin{itemize}
\item $(i)\Rightarrow (iii)$ with $D=\kappa h^2$ and $l_o = c/h$,
\item $(iii) \Rightarrow (ii)$ with $C=2D$,
\item $(i) \Rightarrow (ii)$  with $C=\kappa h^2$.
\end{itemize}
\end{The}

The equivalence $(i) \Leftrightarrow (ii)$ goes back to Bobkov and  G\"otze  \cite{BG99b}. Hence, Theorem \ref{thm:equiv-Poincare} completes the picture by showing that $(i)/(ii)$ also  characterize the measures satisfying  a weak transport-entropy inequality  with a cost function which is quadratic near zero and then linear (like $\theta_1$). The dependence between the constants in the implication $(ii) \Rightarrow (i)$ is not given for technical reasons. Indeed, the proof relies on an argument from \cite{BG99b} that uses a non trivial proof from \cite{BH00} where one loses the explicit dependence on the constants.

We indicate that during the preparation of this work, we learned that the characterization of the convex Poincaré inequality in terms of the convex $(\tau)$-Property (which is equivalent to the transport-entropy inequalities of Item $(iii)$, by duality, as we shall recall in Lemma~\ref{Lem:equiv-inf-conv}) was obtained by Feldheim, Marsiglietti, Nayar and Wang in a recent paper \cite{FMNW15}.

The proof of Theorem \ref{thm:equiv-Poincare} is given in Section \ref{sectionpoincare}. It uses  results of independent interest like a new discrete logarithmic Sobolev inequality for the exponential measure $\tau$ (Lemma \ref{lem:expo}). By transportation techniques, such a logarithmic-Sobolev inequality provides logarithmic-Sobolev inequalities restricted to the class of convex or concave  functions for measures satisfying the condition in Item $(i)$ (Corollary \ref{LSM}). Then  the  weak transport-entropy inequalities of Item $(iii)$ are obtained in their dual forms, involving infimum convolution operators (Lemma~\ref{Lem:equiv-inf-conv}), by means of the Hamilton-Jacobi semi-group approach of Bobkov, Gentil and Ledoux \cite{BGL01}, an approach also generalized in \cite{LV07,GRS14,GRST15}. 
 
\medskip

The paper is organized as follows. In the next section, we introduce and recall some known properties of the convex ordering
 that we shall use in Section \ref{propopti} to prove Theorem \ref{The:OT}.
 Then, in Section \ref{dualc}, we very briefly recall the dual formulation of the weak-transport entropy inequalities $\overline{\mathrm{T}}^\pm$ and $\overline{\mathrm{T}}$, borowed from \cite{GRST15}, which will be useful later on.
Finally, Section \ref{sectionpoincare}  and \ref{sectionfinal} are devoted to the proofs of Theorem \ref{thm:equiv-Poincare} and \ref{thefinal}, respectively.

\medskip

\textbf{Acknowledgment:} The authors would like to warmly thank Greg Blekherman for discussions on the geometric aspects in Section \ref{sec:geomaspect}, including his help with the proof of Theorem \ref{goodprop}.

\section{Convex ordering and a majorization theorem}\label{ordering}

This section is devoted to the study of the convex ordering.  
After recalling some classical definitions and results, we shall prove a majorization theorem which will be a key ingredient in the proof of Theorem \ref{The:OT}.


\subsection{A reminder on convex ordering and the Strassen Theorem}
We collect here some basic facts about convex ordering of probability measures. We refer the interested reader to \cite{Majorization} and \cite{Hirsch} for further results and bibliographic references. All the proofs are well-known, we state some of them for completeness.

We start with the definition of the convex order.

\begin{Def}[Convex order] Given $\nu_1,\nu_2 \in \mathcal{P}_1(\RR)$,  we say that $\nu_2$ dominates $\nu_1$ in the convex order,
and write $\nu_1\preceq\nu_2$, if for all convex functions $f$ on $\RR$,
$\int_\RR f\,d\nu_1\leq \int_\RR f\,d\nu_2$.
\end{Def}

\begin{Rem}
Observe that for any probability measure belonging to $\mathcal{P}_1(\RR)$ the integral of any convex function always makes sense in $\RR\cup\{+\infty\}$. 
\end{Rem}

The convex ordering of probability measures can be determined by testing only some restricted classes of convex functions as the following proposition indicates.
\begin{Pro}\label{equiv-conv-order}
Let $\nu_1,\nu_2 \in \mathcal{P}_1(\RR)$ ; the following are equivalent
\begin{enumerate}
\item[(i)] $\nu_1 \preceq \nu_2$,
\item[(ii)] $\int x\,\nu_1(dx) = \int x \, \nu_2(dx)$ and for all Lipschitz and non-decreasing and non-negative convex function 
$f \colon \RR \to \RR^+$, $\int f(x)\,\nu_1(dx) \leq \int f(x)\,\nu_2(dx)$.
\item[(iii)] $\int x\,\nu_1(dx) = \int x \, \nu_2(dx)$ and for all $t \in \RR$, $\int [x-t]_+\,\nu_1(dx) \leq \int [x-t]_+\,\nu_2(dx)$.
\end{enumerate}
\end{Pro}

For the reader's convenience and for the sake of completeness, we sketch the proof of this classical result. We refer to \cite{Majorization} for more details.

\begin{proof}[Sketch of the proof]
Let us show that $(i)$ is equivalent to $(ii)$. First, since the functions $x\mapsto x$ and $x\mapsto -x$ are both convex, it is clear that $\nu_1 \preceq \nu_2$ implies $\int x\,\nu_1(dx) = \int x \, \nu_2(dx)$ so that $(i)$ implies $(ii)$.
Conversely, since the graph of a convex function always lies above its tangent, subtracting an affine function if necessary, one can restrict to non-negative convex functions. 
Moreover, if $f \colon \RR \to \RR^+$ is a convex function, then $f_n\colon \RR \to \RR^+$ defined by $f_n = f$ on $[-n,n]$, $f_n(x) = f_n(n) + f'_n(n)(x-n)$ if $x\geq n$ and $f_n(x) = f_n(-n) + f'_n(-n)(x+n)$ if $x\leq -n$ 
(where $f_n'$ denotes the right derivative of $f$) is Lipschitz and converges monotonically to $f$ as $n$ goes to infinity. The monotone convergence Theorem then shows that one can further restrict to {\em Lipschitz} convex functions. Finally, up to the subtraction of an affine map, any Lipschitz convex function is non-decreasing, proving that $(ii)$ implies $(i)$.

Now it is not difficult to check that any convex, non-decreasing Lipschitz function
 $f \colon \RR \to \RR^+$ can be approached by a non-increasing sequence of functions of the form $\alpha_0+\sum_{i=1}^n \alpha_i[x-t_i]_+$, with $\alpha_i \geq0$ and $t_i \in \RR$. This shows that $(ii)$ and $(iii)$ are equivalent.
\end{proof}

The next classical result, due to Strassen \cite{Str65}, characterizes the convex ordering in terms of martingales.

\begin{The}[Strassen]\label{Strassen}
Let $\nu_1,\nu_2 \in \mathcal{P}_1(\RR)$ ; the following are equivalent:
\begin{enumerate}
\item[(i)] $\nu_1 \preceq \nu_2$,
\item[(ii)] there exists a martingale $(X,Y)$ such that $X$ has law $\nu_1$ and $Y$ has law $\nu_2$.
\end{enumerate}
\end{The}
We refer to \cite{GRST15} for a (two-line) proof of Theorem~\ref{Strassen} involving 
Kantorovich duality for transport costs of the form $\overline {\mathcal{T}}$.

\subsection{Majorization of vectors and the Rado Theorem}
The convex ordering is closely related to the notion of majorization of vectors that we  recall in the following definition. As for the previous subsection, all the proofs are well-known and we state them for completeness.

\begin{Def}[Majorization of vectors]
Let $a,b \in \RR^n$ ; one says that $a$ is {\em majorized} by $b$, if the sum of the largest $j$ components of $a$ is less than or equal to the corresponding sum of $b$, for every $j$, and if the total sum of the components of both vectors are equal. 
\end{Def}

Assuming that the components of $a=(a_1,\dots,a_n)$ and $b=(b_1,\dots,b_n)$ are in non-decreasing order (\textit{i.e.}\ $a_1 \leq a_2 \leq \dots \leq a_n$ and $b_1 \leq b_2 \leq \dots \leq b_n$), $a$ is majorized by $b$, if
\[
a_n + a_{n-1} + \cdots + a_{n-j+1} \le b_n + b_{n-1} + \cdots + b_{n-j+1}, \qquad \mbox{for } j=1, \dots,n-1,
\]
and $\sum_{i=1}^n a_i = \sum_{i=1}^n b_i$.
\medskip

The next proposition recalls the link between majorization of vectors and convex ordering.

\begin{Pro}\label{maj=conv}
Let $a,b \in \RR^n$ and set $\nu_1 = \frac{1}{n}\sum_{i=1}^n \delta_{a_i}$ and $\nu_2=\frac{1}{n}\sum_{i=1}^n \delta_{b_i}$. The following are equivalent
\begin{enumerate}
\item[(i)] $a$ is majorized by $b$,
\item[(ii)] $\nu_1$ is dominated by $\nu_2$ for the convex order. In other words, for every convex $f:\RR \to \RR$, it holds that $\sum_{i=1}^n f(a_i) \le \sum_{i=1}^n f(b_i)\,.$
\end{enumerate}
\end{Pro}

Thanks to the above proposition and with a slight abuse of notation, in the sequel we will also 
write $a \preceq b$ when $a$ is majorized by $b$.

\begin{proof}
Assume without loss of generality that the components of $a$ and $b$ are sorted in increasing order.
We observe first that, by construction, the equality $\int x \nu_1(dx) = \int x \nu_2(x)$ is equivalent to
$\sum_{i=1}^n a_i = \sum_{i=1}^n b_i$. 

We will first prove that $(i)$ implies $(ii)$. 
By Item $(iii)$ of Proposition \eqref{equiv-conv-order} we only need to prove that $a \preceq b$ implies 
\begin{equation} \label{samedi}
\sum_{k=1}^n [a_k - t]_+ \leq \sum_{k=1}^n [b_k-t]_+, \qquad \forall t \in \RR.
\end{equation}
Assume that $t \leq \max a_k$ (otherwise \eqref{samedi} obviously holds).
Then, let $k_o$ be the smallest $k$ such that $a_k \geq t$ so that $\sum_{k=1}^n [a_k - t]_+  = \sum_{k=k_o}^n (a_k - t)$.
Therefore, by the majorization assumption (which guarantees that $\sum_{k=k_o}^n a_k \leq \sum_{k=k_o}^n b_k$), we get
\begin{align*}
\sum_{k=1}^n [a_k - t]_+  = \sum_{k=k_o}^n (a_k - t) 
 \leq \sum_{k=k_o}^n b_k -t
 \leq \sum_{k=1}^n [b_k-t]_+.
\end{align*}
Conversely, let us prove that $(ii)$ implies $(i)$. Fix $k \in \{1,\ldots,n\}$ and set $f_k(x): =[x-b_k]_+$, $x \in \RR$. Plugging $f_k$ into Item $(ii)$ of Proposition \eqref{equiv-conv-order} leads to
\[
\sum_{i=k}^n a_i-b_k \leq \sum_{i=1}^n [a_i-b_k]_+ = n \!\int\! f(x) \nu_1(dx) \leq n \!\int\! f(x) \nu_2(dx) = \sum_{i=1}^n [b_i-b_k]_+ = \sum_{i=k}^n b_i-b_k,
\]
so that $\sum_{i=k}^n a_i \leq \sum_{i=k}^n b_i$, which proves that $a$ is majorized by $b$.
\end{proof}

Next we recall a simple classical consequence of Proposition \ref{maj=conv} in terms of discrete optimal transport on the line.

\begin{Pro}\label{OTdim1}
Let $x,y \in \RR^n$ be two vectors whose coordinates are listed in non-decreasing order (\textit{i.e.}\ 
$x_1 \leq x_2 \leq \dots\leq x_n$, $y_1 \leq y_2 \leq \dots \leq y_n$).
Then for all permutation $\sigma$ of $\{1,\ldots,n\}$ and all convex $\theta : \RR \to \RR$, it holds
\[
\sum_{i=1}^n \theta (x_i-y_i) \leq \sum_{i=1}^n \theta(x_i - y_{\sigma(i)}).
\]
\end{Pro}
\begin{proof}
Since, for all $k$, $\sum_{i=k}^n y_i \geq \sum_{i=k}^n y_{\sigma(i)}$, it holds for $\sum_{i=k}^n (x_i - y_i) \leq \sum_{i=k}^n (x_i - y_{\sigma(i)})$ (with equality for $k=1$). Therefore, denoting $y_{\sigma}=(y_{\sigma(1)},\ldots,y_{\sigma(n)})$, it holds $x-y \preceq x - y_{\sigma}.$ Applying Proposition \ref{maj=conv} completes the proof.
\end{proof}

\begin{Rem}\label{RemOT}In particular, let $\mu,\nu$ are two discrete probability measures on $\RR$ of the form 
\[
\mu= \frac{1}{n}\sum_{i=1}^n \delta_{x_i}\qquad \text{and} \qquad \nu = \frac{1}{n}\sum_{i=1}^n \delta_{y_i},
\]
where the $x_i$'s and the $y_i$'s are in increasing order, and assume for simplicity that the $x_i$'s
are distinct. Then the map $T$ sending $x_i$ on $y_i$ for all $i$ realizes the optimal transport of $\mu$ onto $\nu$ for every cost function $\theta.$
\end{Rem}
\medskip

We end this section with a characterization of the convex ordering (or equivalently of the majorization of vectors, thanks to Proposition \ref{maj=conv}), due to Rado \cite{Rad52}. We may give a proof based on Strassen's Theorem. 
For simplicity, we denote by $\mathcal{S}_n$ the set of all permutations of $\{1,2, \ldots,n\}$ and, given $\sigma \in \mathcal{S}_n$
and $x=(x_1,\dots,x_n) \in \RR^n$, we set $x_\sigma:=(x_{\sigma(1)},\dots,x_{\sigma(n)})$.
\begin{The}[Rado]\label{Rado}
Let $a,b \in \RR^n$ ; the following are equivalent
\begin{enumerate}
\item[(i)] the vector $a$ is majorized by $b$;
\item[(ii)] there exists a doubly stochastic matrix $P$ such that $a=bP$; 
\item[(iii)] there exists a collection of non-negative numbers $(\lambda_\sigma)_{\sigma \in S_n}$ with $\sum_{\sigma \in \mathcal{S}_n} \lambda_\sigma = 1$ such that $a= \sum_{\sigma \in \mathcal{S}_n} \lambda_\sigma b_\sigma$ 
(in other words $a$ lies in the convex hull of the permutations of $b$).
\end{enumerate}
\end{The}
\proof
First we will prove that $(i)$ implies $(ii)$. According to Proposition \ref{maj=conv}, $a\preceq b$ is equivalent to saying that $\nu_1 = \frac{1}{n}\sum_{i=1}^n \delta_{a_i}$ is dominated by $\nu_2 = \frac{1}{n} \sum_{i=1}^n \delta_{b_i}$ in the convex order. Set $\mathcal{X}:=\{a_1,\ldots,a_n\}$, $\mathcal{Y}:=\{b_1,\ldots,b_n\}$,
$k_x:=\# \{ i \in \{1,\ldots,n\} : a_i =x\}$, $x \in \mathcal{X}$ and $\ell_y=\#\{i \in \{1,\ldots,n\} : b_i = y\}$, $y \in \mathcal{Y}$ (where $\#$ denotes the cardinality);  observe that
 $\nu_1 = \frac{1}{n}\sum_{x \in \mathcal{X}} k_x \delta_x$ and $\nu_2=\frac{1}{n}\sum_{y \in \mathcal{Y}} \ell_y \delta_y$.
According to the Strassen Theorem (Theorem~\ref{Strassen}), there exists a couple of random variables $(X,Y)$ on some probability space $(\Omega,\mathcal{A},\P)$ such that $X$ is distributed according to $\nu_1$, and $Y$ according to $\nu_2$ and $X = \E[Y|X].$
Since $X$ is a discrete random variable, 
\[
\E[Y|X] = \sum_{x \in \mathcal{X}} \frac{\E[Y \mathbf{1}_{X=x}]}{\P(X=x)}\mathbf{1}_{X=x},\qquad \text{a.s.}
\] 
Therefore, for all $x \in \mathcal{X}$,
\[
x = \frac{\E[Y \mathbf{1}_{X=x}]}{\P(X=x)} = \sum_{y \in \mathcal{Y}} \ell_y yK_{y,x},
\]
where  $K_{y,x} := n\frac{\P(X=x, Y=y)}{k_x\ell_y}.$ Hence $a=bP$ with $P_{j,i} := K_{b_j,a_i}$, $i,j=1,\dots,n$.
This proves Item $(ii)$, since $P$ is doubly stochastic by construction.

If $a=bP$ with a doubly stochastic matrix $P$, then it is easily checked that $\sum_{i=1}^n f(a_i) \leq \sum_{i=1}^n f(b_i)$ for 
any convex function $f$ on $\RR$  so that $(ii)$ implies $(i)$.

Finally, according to Birkhoff's theorem, the extremes points of the set of doubly stochastic matrices are \emph{permutation matrices}. Therefore every doubly stochastic matrix can be written as a convex combination of permutation matrices showing that $(ii)$ and $(iii)$ are equivalent.
\endproof

\subsection{Geometric aspects of convex ordering and a majorization theorem}\label{sec:geomaspect}

Contrary to the previous subsections, the results presented here are new.
Fix some vector $b = (b_1,b_2, \ldots ,b_n)$ of  $\RR^n$  with {\em distinct} components (for simplicity).
We will be working with the convex hull of the permutations of $b$, a polytope we denote by 
$\mathrm{Perm}(b)$ and defined as
\[
\mathrm{Perm}(b):= \left\{ \sum_{\sigma \in \mathcal{S}_n} \lambda_\sigma b_\sigma,  \mbox{with } \lambda_\sigma \geq 0 \mbox{ and } \sum_{\sigma \in S_n} \lambda_\sigma = 1 \right\} .
\]
Such a polytope is often refered to as the \emph{Permutahedron} generated by $b$. According to Rado's Theorem \ref{Rado}, $\mathrm{Perm}(b)=\{a \in \RR^n : a \preceq b\}.$
Hence, $\mathrm{Perm}(b)$ is a subset of the following affine hyperplane
\[
\mathcal{E}_b:=\left\{x \in \RR^n : \sum_{i=1}^n x_i=\sum_{i=1}^n b_i\right\} = b + \mathcal{E}_0,
\]
with $\mathcal{E}_0 := \{x \in \RR^n : \sum_{i=1}^n x_i =0\}.$

We will be interested in the faces, facets containing a given face, and normal vectors to such facets of $\mathrm{Perm}(b)$. We need to introduce some notations.

Denote by $[n]$ the set of integers from $1$ to $n$. For $S\subset [n]$, let $v_S(b)$ denote the vector with
the $|S|$ largest components of $b$ in the positions indexed by $S$ (in decreasing order, say), and the remaining $n-|S|$ lowest components of $b$ in the other positions indexed by $[n]\setminus S$ (also in a decreasing order).
Also, when $S \neq \emptyset$, we denote by $P_S(b)$ the set that contains the vector $v_S(b)$ along with all vectors obtained by permuting any subset of coordinates of $v_S(b)$, as long as the subset is contained in $S$ or  in $[n]\setminus S$. (That is, the only permutations that are {\em not} allowed are those that involve elements from {\em both} $S$ and $[n]\setminus S$). More precisely
$$
P_S(b) : = \{ (v_S(b))_\sigma, \sigma \in \mathcal{S}_n \mbox { such that } \sigma(S) = S \} 
$$
where $\sigma(\mathcal{S}):=\{\sigma(i), i \in \mathcal{S}\}$ denotes the image of $\mathcal{S}$ by $\sigma$.

More generally, given a partition ${\mathcal{S}} = (S_1, S_2, \ldots , S_k)$ of $[n]$, let $v_{\mathcal{S}}(b)$ denote the vector with the largest $|S_1|$ coordinates of $b$ in the positions indexed by $S_1$ (in decreasing order), then the next largest $|S_2|$ coordinates in the positions indexed by $S_2$ and so on (as an illustration, for $b=(1,4,5,-2,3,9,6,-5) \in \mathbb{R}^8$ and ${\mathcal{S}} = (S_1, S_2, S_3)$ with
$S_1=\{1,2\}$, $S_2=\{3,6,7\}$ and $S_3=\{4,5,8\}$, we get 
$v_{\mathcal{S}}(b)=(\mathit{9},\mathit{6},\hm{5}, 1,-2 ,\hm{4}, \hm{3}, -5)$ where the italic positions refer to the set $S_1$, the bold positions to the set $S_2$ and the remaining positions to $S_3$).
Also, we denote by $P_{\mathcal{S}}(b)$ the set containing the vector $v_{\mathcal{S}}(b)$ along with all vectors obtained by permuting the coordinates of $v_{\mathcal{S}}(b)$ that belong to the same $S_i$:
$$
P_{\mathcal{S}}(b) : = \{ (v_S(b))_\sigma, \sigma \in \mathcal{S}_n \mbox { such that for all }i, \sigma(S_i) = S_i  \}.
$$

Now we recall two geometric definitions/facts from \cite{Matroids}.

\begin{itemize}
\item[Fact 1:] A {\em facet} of $\mathrm{Perm}(b)$ is the convex hull of $P_S(b)$, for some $S \neq \emptyset,  [n]$.

\item[Fact 2:] A {\em face} of $\mathrm{Perm}(b)$ is the convex hull of $P_{\mathcal{S}}(b)$, for some partition $\mathcal{S} = (S_1, S_2, \ldots , S_k)$ of $[n]$  with $k \geq 3$.  
Furthermore, given a face $F = \mathrm{Conv} ( P_{\mathcal{S}}(b) )$, there exist exactly $k-1$ facets 
containing $F$ that are obtained by coalescing the first and last several $S_i$'s in $\mathcal{S}$: that is, for each $1\le j\le k-1$,  the facet $F_j$ containing $F$ can be described by taking the partition $[n] = T_1 \cup T_2$
with $T_1 = S_1 \cup \cdots \cup S_j$, and $T_2 = S_{j+1} \cup \cdots \cup S_k$.
\end{itemize}

The next theorem, which we may call the Majorization Theorem, is a key ingredient in the proof of Theorem \ref{The:OT}.
It provides a geometric interpretation of majorization in terms of projection.

\begin{The}[Majorization Theorem]\label{goodprop}
Let $a,b \in \RR^n$, assume that $b$ has distinct coordinates and that $a \notin \mathrm{Perm}(b)$. Then the following are equivalent:
\begin{itemize}
\item[(i)] $\hat{c} \in \mathrm{Perm}(b)$ satisfies
\[ a-\hat{c} \preceq a-c , \ \ \ \forall c \in \mathrm{Perm}(b)\,;\]

\item[(ii)] $\hat{c}$ is the closest point of $\mathrm{Perm}(b)$ to $a$; that is,
\[\hat{c}:=\arg\min_{c \in \mathrm{Perm}(b)}(\|a-c\|_2)\,.\]
\end{itemize}
Moreover the vector $\hat{c}$ is sorted as $a$ : $(a_i \leq a_j) \Rightarrow (\hat{c}_i \leq \hat{c}_j)$ , for all $i,j$.
\end{The}

Let us recall that the orthogonal projection of a point $a$ on the polytope $\mathrm{Perm}(b)$ is the unique $\bar{c} \in \mathrm{Perm} (b)$ such that
\begin{equation}\label{eq:proj}
\langle a-\bar{c}, c - \bar{c}\rangle \leq 0,\qquad \forall c \in \mathrm{Perm}(b).
\end{equation}

\begin{proof} 
Observe that if $\sum_{i=1}^n a_i \neq \sum_{i=1}^n b_i$, then letting 
 $\tilde{a}: = a - \frac{k}{n} (1,1,\ldots,1)$ with $k:= \sum_{i=1}^n a_i -\sum_{i=1}^n b_i$, we see (using \eqref{eq:proj}) that the orthogonal projection of $a$ and $\tilde{a}$ on $\mathrm{Perm}(b)$ are equal (to some point we denote by $\hat{c}$, say), and that $a-\hat{c} \preceq a-c$  if and only if $\tilde{a}-\hat{c} \preceq \tilde{a}-c$.
Therefore we can assume without loss that $a$ and $b$ are such that $\sum_{i=1}^n a_i = \sum_{i=1}^n b_i$. 

We will first prove that $(i)$ implies $(ii)$ which is the easy part of the proof.

\smallskip

\noindent $(i) \implies (ii)$.
Let $\bar{c}$ be the closest point of $\mathrm{Perm}(b)$ to $a$ (\textit{i.e.}\ $\bar{c}:=\arg\min_{c \in \mathrm{Perm}(b)}(\|a-c\|_2)$). Then by $(i)$, $a-\hat{c} \preceq a-\bar{c}$, which, by Proposition~\ref{maj=conv} (applied to $f(x)=x^2$) 
implies that $\sum_{i=1}^n(a-\hat{c})_i^2 \le  \sum_{i=1}^n (a-\bar{c})_i^2$. By definition of $\bar{c}$, this is possible only if $\hat{c} = \bar{c}$.

\medskip

Next we will prove that $(ii)$ implies $(i)$. For the sake of clarity, we first deal with the simple case when 
$\hat{c}$ lies on a {\em facet} of $\mathrm{Perm}(b)$, before dealing with the general case of $\hat{c}$ being on a face.

\smallskip

\noindent $(ii) \implies (i)$.
Let $\hat{c}$ be the closest point of $\mathrm{Perm}(b)$ to $a$. Since $\mathrm{Perm}(b)$ is invariant by permutation, it easily follows from 
Proposition \ref{OTdim1} that the coordinates of $\hat{c}$ are in the same order as the coordinates of $a$. 
Hence, we are left with the proof that $a- \hat{c} \preceq a-c$ for all $c \in \mathrm{Perm}(b)$.

\noindent
{\bf (a) A simple case: $\hat{c} \in F$ for some facet $F$}. 
Since $\hat{c}$ is chosen from $\mathrm{Perm}(b)$, and since we assumed  that
$\sum_i b_i = \sum_i a_i$, we have $\sum_i (a-\hat{c})_i = 0$.
Writing $\alpha:=a-\hat{c} \in \mathcal{E}_0$, suppose that $\alpha$ is perpendicular to the affine subspace $\mathcal{H}:=\mathcal{H}_F$ containing a facet $F$, defined by some nonempty subset $S$ of $[n]$. For all $x,y \in F$, we thus have $\langle \alpha, x-y \rangle =0$. Choosing $x=v_S(b)$ and $y=x_{\tau_{ij}}$ obtained by permuting two coordinates of $x$ whose indices are both in $S$ or both in $S^c$ (\textit{i.e.}\ $\tau_{ij}=(ij)$ is the transposition that permutes $i$ and $j$, with $i,j \in S$, or $i,j \in S^c$), one sees that the coordinates of $\alpha$ are constant on $S$ and $S^c$. We denote by $\alpha_S$ and $\alpha_{S^c}$ the values of $\alpha$ on these sets, which verify $k\alpha_S + (n-k)\alpha_{S^c}=0$ 
since $\alpha \in \mathcal{E}_0$.

Now (recalling that $\alpha=a -\hat{c}$) our task is to show  that
\begin{equation*}
\alpha \preceq  \alpha - (c' - \hat{c}),\ \mbox{ for every $c' \in \mathrm{Perm}(b)$}\,.
\end{equation*}
This amounts to showing that
\begin{equation*}
\alpha \preceq \alpha - c,\ \mbox{ for every $c$ such that }
\langle\alpha, c\rangle \leq 0, \ \mbox{ and }  \sum_i c_i = 0\,.
\end{equation*}
Indeed, on the one hand the choice of $\hat{c}$ implies, by \eqref{eq:proj}, that for every $c' \in \mathrm{Perm}(b)$
$\langle\alpha, c' - \hat{c}\rangle \ \le 0$, and on the other hand, since $\hat{c}, c' \in \mathrm{Perm}(b)$, necessarily $\sum_i c_i = \sum_i c'_i -  \sum_i  \hat{c}_i = 0$.

Now $\langle\alpha, c\rangle \leq 0$ and $\sum_i c_i = 0$ together imply (recall that $\alpha$ is constant on $S$ and $S^c$) that
\[
(\alpha_S - \alpha_{S^c})\sum_{i \in S} c_i \leq 0\,.
\]
Let us assume that $\alpha_S >\alpha_{S^c}$. Then denoting by $c_S = \sum_{i\in S}c_i$ and by $c_{S^c}= \sum_{i\in S^c} c_i$, one has $c_S\leq 0$ and $c_{S^c} \geq0.$
Therefore, for any convex function $f$ on $\RR$, according to Jensen's inequality and by convexity, we get
\begin{align*}
\sum_{i=1}^n f(\alpha_i-c_i) &= k \frac{\sum_{i\in S} f(\alpha_S - c_i)}{k} + (n-k) \frac{\sum_{i\in S^c}f(\alpha_{S^c} - c_i) }{n-k}\\
& \geq k f\left(\alpha_S - \frac{c_S}{k}\right) + (n-k)f\left(\alpha_{S^c}- \frac{c_{S^c}}{n-k}\right)\\
& \geq kf(\alpha_S)+(n-k) f(\alpha_{S^c}) -f'(\alpha_S)c_S -f'(\alpha_{S^c})c_{S^c} \\
& \geq \sum_{i=1}^n f(\alpha_i),
\end{align*}
where the last inequality comes from the fact that 
$f'(\alpha_S)c_S + f'(\alpha_S^c)c_{S^c} = c_S(f'(\alpha_S) - f'(\alpha_{S^c})) \leq 0$.
According to Proposition \ref{maj=conv}, we conclude that $\alpha \preceq \alpha - c$ which is the expected result.

\

\noindent
{\bf (b) The general case.} Suppose that $\hat{c}$ lies in a face $F$ of the polytope. This face is related to a partition $\mathcal{S} = (S_1,\ldots,S_k)$ of $[n]$, with $k \geq 3$. Then $\alpha:=a-\hat{c} \in N(F)$, where $N(F)$ denotes the normal cone of $F$.  Recall that the extreme rays of $N(F)$ are given by the facet directions for the facets containing $F$. For all $i \in \{1,\ldots,n-1\}$, let us denote by $F_i$ the facet containing $F$ associated to the partition $\mathcal{T}_i=\{S_1\cup\ldots \cup S_{i} \,;\,  S_{i+1}\cup\ldots\cup S_k\}$, $1\leq i\leq k-1.$  Consider the vectors $p_1, p_2, \ldots , p_{k-1} \in \mathcal{E}_0$  defined by
\[
p_i = \bone_{S_1 \cup S_2 \cup \cdots \cup S_i} - \frac{k_i}{n}\bone_{[n]}
\]
where $\bone_T$ denotes the $0-1$ indicator vector of $T$, for $T\subseteq [n]$, and $k_i = |S_1|+\cdots+|S_i|$.
For each $i$, the vector $p_i$ is orthogonal to the facet $F_i$. Moreover, for all $c \in \mathrm{Perm}(b)$ one may check that $ \langle c,p_i \rangle \leq \langle v_{\mathcal{T}_i},p_i \rangle$, with equality on $F_i$. This shows that $p_i$ is an outward normal vector to $F_i$. Therefore $N(F)$ is the conical hull of the $p_i$'s, and so we may express $\alpha$, for a suitable choice of  $\lambda_i \ge 0$, as:

\[\alpha = \sum_i \lambda_i \bone_{S_1 \cup S_2 \cup \cdots \cup S_i} - \sigma \bone_{[n]} \,,\]
where $\sigma = (1/n)[\sum_{i=1}^{k-1} \lambda_i |S_1| +
\sum_{i=2}^{k-1} \lambda_i |S_2| + \cdots +  \lambda_{k-1}|S_{k-1}|]$\,.
In particular, $\alpha$ is constant on each $S_j$ : for all $i\in S_j$, $\alpha_i = \left(\sum_{p=j}^{k-1} \lambda_p\right)-\sigma:=A_j.$

In order to establish (i), we  need to show that
\[ \alpha \preceq \alpha - (c - \hat{c}) , \ \ \ \forall c \in \mathrm{Perm}(b)\,,\] or in other words, we need to show that
\[ \alpha \preceq \alpha - c' , \ \ \ \forall c' \in \mathrm{Perm}(b) - \hat{c}\,.\]

\noindent
We now use again the fact that our choice of $\hat{c}$ implies that, for all $1\le i\le k-1$,
\begin{equation*}
\langle p_i, \hat{c}\rangle  \ge   \langle p_i, c\rangle\,, \ \ \forall c \in \mathrm{Perm}(b)\,.
\end{equation*}
This in turn gives the following:
\[  \mathrm{Perm}(b) - \hat{c} \subseteq   \{c' :  \langle c', p_i\rangle   \le  0, \ \forall i\}\,.\]
Thus using $N(F)^0 := \{d\in \mathcal{E}_0 ;  \langle d, p_i\rangle  \ \le \ 0, \ \forall i\}$ to denote the {\em polar cone}, it then suffices to show that for  $\alpha$ (as above),
\[ \alpha \preceq \alpha - d, \ \ \forall d \in N(F)^0\,.\]
Now, $d \in N(F)^0$ implies that
\[ 
\langle d, \bone_{S_1 \cup S_2 \cup \cdots \cup S_j}\rangle   \le  0\,  \mbox{ and }  \sum_i d_i =0\,,
\]
therefore denoting $E_j = \sum_{i\in S_1\cup\ldots\cup S_j} d_i$, for all $j\in \{0,1,\ldots,k\}$, one has $E_j \leq 0$ and $E_0=E_k=0.$

Let $f: \RR \to \RR$ be a convex function ; denoting by $f'$ its right derivative, the convexity of $f$ implies that
\[
\sum_{i=1}^n f(\alpha_i-d_i)  = \sum_{j=1}^k \sum_{i \in S_j} f(A_j - d_i)
\geq \sum_{j=1}^k |S_j|f(A_j) - \sum_{j=1}^k f'(A_j)D_j,
\]
where $D_j = \sum_{i\in S_j} d_i.$ Now, using an Abel transform (and the fact that $E_0=E_k=0$), one gets
\[
\sum_{j=1}^k f'(A_j)D_j = \sum_{j=1}^k f'(A_j)(E_j-E_{j-1}) = \sum_{j=1}^{k-1} (f'(A_j)-f'(A_{j+1})E_j \leq 0,
\]
where the inequality comes from $E_j \leq 0$, $A_j \geq A_{j+1}$ and the monotonicity of $f'.$
Therefore, one gets
\[
\sum_{i=1}^n f(\alpha_i-d_i) \geq \sum_{j=1}^k |S_j|f(A_j) = \sum_{i=1}^n f(a_i),
\]
which proves that $a \preceq a-d$, thanks to Proposition \ref{maj=conv}, as expected. This completes the proof.
\end{proof}

\section{Properties of the optimal coupling for weak transport costs}\label{propopti}

This section is devoted to the proof of Theorem  \ref{The:OT}. We will establish first a preliminary result which gives 
some connection between $\mathcal{T}$ and $\overline{\mathcal{T}}$. 
In the sequel, we denote by $\mathrm{Im}(\mu)$, respectively $\mathrm{Im}^{\uparrow}(\mu)$, the set of probability measures on $\RR$ which are images of $\mu$ under some map $S:\RR \to \RR$, respectively some non-decreasing map $S$, \textit{i.e.,} 
\[
\mathrm{Im}(\mu) = \{ \gamma \in \mathcal{P}(\RR) : \exists S : \RR \to \RR \text{ measurable such that } \gamma = S_\# \mu \},
\]
and 
\[
\mathrm{Im}^{\uparrow}(\mu) = \{ \gamma \in \mathcal{P}(\RR) : \exists S : \RR \to \RR \text{ measurable, non-decreasing, such that } \gamma = S_\# \mu \} .
\] 
\begin{Pro}\label{pronu1}
For all probability measures $\mu,\nu$ on $\RR$, it holds
\[
\inf_{\gamma\preceq\nu, \gamma \in \mathrm{Im}^{\uparrow}(\mu)}\mathcal{T}_\theta(\gamma,\mu) \geq 
\overline{\mathcal{T}}_\theta(\nu|\mu)\geq \inf_{\gamma\preceq\nu, \gamma \in \mathrm{Im}(\mu)}\mathcal{T}_\theta(\gamma,\mu).
\]
\end{Pro}
\medskip

\begin{Rem} \label{rem}
Note that when $\mu$ has no atoms, then $\mathrm{Im}^{\uparrow}(\mu) = \mathrm{Im}(\mu)$.
If $\mu$ is a discrete probability measure, then the two sets may be different. For instance, if $\mu = \frac{1}{3}\delta_0 + \frac{2}{3}\delta_1$, then $\gamma = \frac{2}{3}\delta_0 + \frac{1}{3}\delta_1$ is in $\mathrm{Im}(\mu)$ but not in $\mathrm{Im}^{\uparrow}(\mu)$. In the proof of Theorem \ref{The:OT} below, we will use 
Proposition \ref{pronu1} with $\mu$ being the uniform distribution on $n$ distinct points for which it is clear that $\mathrm{Im}^{\uparrow}(\mu) = \mathrm{Im}(\mu)$.
\end{Rem}

\begin{proof}
First we will prove that $\overline{\mathcal{T}}_\theta(\nu|\mu)\geq \inf_{\gamma\preceq\nu, \gamma \in \mathrm{Im}(\mu)}\mathcal{T}_\theta(\gamma,\mu)$.
To that aim, denote by  $\pi(dxdy) = p(x,dy)\mu(dx)$ some coupling between $\mu$ and $\nu$ and set $S (x) := \int y \,p(x,dy)$, $x\in \RR$. Clearly $S_\#\mu \in \mathrm{Im}(\mu)$. Moreover if $f : \RR \to \RR$ is some convex function, by Jensen's inequality, it holds
\begin{align*}
\int f(x)\, S_\#\mu (dx) = \int f\left( \int y\,p(x,dy)\right)\,\mu(dx) 
\leq \iint f(y)\,p(x,dy)\mu(dx) = \int f(y)\,\nu(dy)
\end{align*}
so that $S_\#\mu \preceq \nu$.
Therefore,
\begin{align*}
\int \theta\left( x - \int y \,p(x,dy)\right)\,\mu(dx) &= \int \theta(x - S(x))\,\mu(dx) \geq \mathcal{T}_\theta(S_\#\mu,\mu)
& \geq \!\!\inf_{\gamma \preceq \nu, \gamma \in \mathrm{Im}(\mu)} \!\!\mathcal{T}_\theta(\gamma,\mu)
\end{align*}
from which the claim follows by taking the infimum over $p$.

Now we turn to the proof of the inequality $\overline{\mathcal{T}}_\theta(\nu|\mu)\leq \inf_{\gamma\preceq\nu, \gamma \in \mathrm{Im}^\uparrow(\mu)}\mathcal{T}_\theta(\gamma,\mu)$.
Assume that $\gamma\preceq \nu$ and that $\gamma = S_\# \mu$ for some non-decreasing map $S$. According to Strassen's theorem, there exists a coupling $\pi_1$ with first marginal $\gamma$ and second marginal $\nu$ such that $\pi_1(dxdy)=p_1(x,dy)\gamma(dx)$ and $x=\int_\RR yp_1(x,dy)$, $\gamma$ almost everywhere. For all $x\in \RR$, define the following probability measure $p(x,dy):=p_1(S(x),dy)$.
 Then for all bounded continuous function $f$, it holds
\begin{align*}
\iint f(y)p(x,dy)\,\mu(dx)
&=
\iint f(y)p_1(S(x),dy)\,\mu(dx)\\
&=\iint f(y) p_1(x,dy)\,\gamma(dx)=\int f(y)\,\nu(dy).
\end{align*}
Thus the coupling $\pi(dxdy)=p(x,dy)\mu(dx)$ has $\mu$ as first marginal and $\nu$ as second marginal. Moreover, by definition of $p_1$ and $p$, $\mu$ almost everywhere, it holds 
\[
\int yp(x,dy) = \int y p_1(S(x),dy) =  S(x).
\]
Since $S$ is non-decreasing, it realizes the optimal transport between $\mu$ and $\nu$ for the classical transport cost $\mathcal{T}_\theta$ and so it follows that
\begin{align*}
\mathcal{T}_\theta(\gamma,\mu)
=
\int\theta(|x-S(x)|)\mu(dx)
=
\int\theta(|x-\int y p(x,dy)|)\mu(dx)\geq \overline{\mathcal{T}}_\theta(\nu|\mu)
\end{align*}
which achieves the proof by taking the infimum over $\gamma$.
\end{proof}

We  are now in a position to prove Theorem \ref{The:OT}. 


\begin{proof}[Proof of Theorem \ref{The:OT}]
The proof of the first part of Theorem \ref{The:OT} is divided into two steps. In the first step we will deal with uniform discrete measures on $n$ points, while in the second step we will use an approximation argument in order to reach any measures.

 \smallskip

\textbf{Step 1.} We  first deal with
\[
\mu := \frac{1}{n}\sum_{i=1}^n \delta_{a_i}\qquad \text{and}\qquad \nu := \frac{1}{n} \sum_{i=1}^n \delta_{b_i},
\]
with $a_1 <a_2<\ldots < a_n$ and $b_1<b_2<\ldots< b_n$.
Set $a:=(a_1,\ldots,a_n)$ and $b:=(b_1,\ldots,b_n)$. 
According to Theorem \ref{goodprop}, there exists some $\hat{c} \in \mathrm{Perm}(b)$ such that $a - \hat{c} \preceq a - c$, for all $c \in \mathrm{Perm}(b)$. Moreover the coordinates of $\hat{c}$ satisfy $\hat{c}_i \leq \hat{c}_{i+1}$. 
Set $\hat{\gamma}:=\frac{1}{n}\sum_{i=1}^n \delta_{\hat{c}_i}$ and observe that $\nu$ dominates $\hat{\gamma}$ for the convex order and $\hat{\gamma} \in \mathrm{Im}^\uparrow(\mu)$. (Recall the definition from the beginning of this section.)

Now for any $\gamma:=\frac{1}{n}\sum_{i=1}^n \delta_{c_i} \in \mathrm{Im}^\uparrow(\mu)$ with  $c_i\leq c_{i+1}$
and for any convex cost function $\theta$,
it holds (since the coordinates are non-decreasing) 
\[
\mathcal{T}_\theta(\gamma,\mu) = \frac{1}{n}\sum_{i=1}^n \theta(|a_i-c_i|).
\]
In particular
\begin{equation}\label{eq:proofThe:OT}
\mathcal{T}_\theta(\hat{\gamma},\mu) 
= 
\frac{1}{n}\sum_{i=1}^n \theta(|a_i-\hat{c}_i|) 
\leq 
\inf_{c \in \mathrm{Perm}(b)}\frac{1}{n}\sum_{i=1}^n \theta(|a_i -c_i|)
.
\end{equation}
A probability $\gamma$ such that $\gamma \preceq \nu, \gamma \in \mathrm{Im}^\uparrow(\mu)$ is of the form $\gamma=\frac{1}{n}\sum_{i=1}^n \delta_{c_i}$ with $c_i\leq c_{i+1}$ and $c=(c_1,\ldots,c_n) \in \mathrm{Perm}(b)$, and for such a $c$, it holds $ \frac{1}{n}\sum_{i=1}^n \theta(|a_i-c_i|)=\mathcal{T}_\theta(\gamma,\mu)$. Therefore, the latter implies
$$
\mathcal{T}_\theta(\hat{\gamma},\mu) 
\leq
\inf_{\gamma \preceq \nu, \gamma \in \mathrm{Im}^\uparrow(\mu)} \mathcal{T}_\theta(\gamma,\mu) =\overline{\mathcal{T}}_\theta(\nu|\mu)
$$
where the last equality follows from Proposition \ref{pronu1} and the fact that for such a distribution $\mu$, it holds $\mathrm{Im}(\mu) = \mathrm{Im}^\uparrow (\mu)$ (see Remark \ref{rem}).
Since obviously $\overline{\mathcal{T}}_\theta(\nu|\mu) \leq \mathcal{T}_\theta(\hat{\gamma},\mu)$, 
 we conclude that $\mathcal{T}_\theta(\hat{\gamma},\mu) = \overline{\mathcal{T}}_\theta(\nu|\mu)$ as expected.

\smallskip

\textbf{Step 2.} In the second step we deal with  the general case using an approximation argument.

Let $\mu$ and $\nu$ be two elements of $\mathcal{P}_1(\RR)$. By assumption, $\int |x|\,\mu(dx)< \infty$  and $\int |x|\,\nu(dx) <\infty$, hence, according to the de la Vall\'ee-Poussin Theorem (see \textit{e.g.} \cite[Theorem 4.5.9]{Bog07}), there exists an increasing convex function $\beta:\RR^+ \to \RR^+$ such that $\beta(t)/t \to \infty$ as $t \to \infty$ and such that $\int \beta(|x|)\,\mu(dx) <\infty$ and $\int \beta(|x|)\,\nu(dx) <\infty.$

Next we will construct discrete approximations of $\mu$ and $\nu$. According to Varadarajan's theorem (see \textit{e.g.} \cite[Theorem 11.4.11]{Dud02}), if $X_i$ is an i.i.d sequence of law $\mu$, then, with probability $1$, the empirical measure $L_n^X := \frac{1}{n} \sum_{i=1}^n \delta_{X_i}$ converges weakly to $\mu$. 
On the other hand, according to the strong law of large numbers, with probability $1$, $\frac{1}{n}\sum_{i=1}^n |X_i| \to \int |x|\,\mu(dx)$ as $n \to \infty$.
Let us take $(x_i)_{i \geq 1}$, a positive realization of these events and set $\mu_n = \frac{1}{n}\sum_{i=1}^n \delta_{x_i^{(n)}},$ where $x_1^{(n)}\leq x_2^{(n)}\leq \ldots\leq x_n^{(n)}$ denotes the increasing re-ordering of the vector $(x_1,x_2,\ldots,x_n).$  Then the sequence $\mu_n$ converges weakly to $\mu$ and $\int |x|\,\mu_n(dx) \to \int |x|\,\mu(dx).$ According to \cite[Theorem 6.9]{Vil09}, this is equivalent to the convergence of the $W_1$ distance : $W_1(\mu_n,\mu) \to 0$ as $n \to \infty.$
Note that one can assume that the points $x_i^{(n)}$ are distinct. Indeed, if this is not the case, then letting $\tilde{x}_i^{(n)}=x_i^{(n)} + i/n^2$ one obtains distinct points and it is not difficult to check that $\tilde{\mu}_n = \frac{1}{n}\sum_{i=1}^n \delta_{\tilde{x}_i^{(n)}}$ still weakly converges to $\mu$ (for instance the $W_1$ distance between $\mu_n$ and $\tilde{\mu}_n$ is easily bounded from above by $(n+1)/(2n^2)$). The same argument yields a sequence $\nu_n = \frac{1}{n}\sum_{i=1}^n \delta_{y_i^{(n)}}$ with $y_i^{(n)}<y_{i+1}^{(n)}$ converging to $\nu$ in the $W_1$ sense. It is not difficult to check (invoking the strong law of large numbers again) that one can further impose that $\int \beta(|x|)\,\nu_n(dx) \to \int \beta(|x|)\,\nu(dx)$, as $n \to \infty.$

For all $n\geq 1$, one applies the result proved in the first step : there exists a unique probability measure $\hat{\gamma}_n \preceq \nu_n$ such that
\[
\overline{\mathcal{T}}_\theta(\nu_n |\mu_n) = \mathcal{T}_\theta(\hat{\gamma}_n,\mu_n),
\]
for all convex cost functions $\theta$. Let us show that one can extract from $\hat{\gamma}_n$ a subsequence converging to some $\hat{\gamma}$ in $\mathcal{P}_1(\RR)$ for the $W_1$ distance. By construction $\int \beta(|x|)\,\nu_n(dx) \to \int \beta(|x|)\,\nu(dx)$ and so $M=\sup_{n\geq1} \int \beta(|x|)\,\nu_n(dx)$ is finite. Since $\hat{\gamma}_n \preceq \nu_n$ and since the function $x \mapsto \beta(|x|)$ is convex, it thus holds $\int \beta(|x|)\,\gamma_n(dx) \leq \int \beta(|x|)\,\nu_n(dx) \leq M.$ In particular, setting $c(R) = \inf_{t\geq R} \beta(t)/t$, $R>0$, Markov's inequality easily implies that 
\[
\int_{[-R,R]^c} |x|\,\hat{\gamma}_n(dx) \leq \frac{\int \beta(|x|)\,\nu_n(dx) }{c(R)}\leq \frac{M}{c(R)}. 
\]
Consider $\tilde{\gamma}_n$ defined by $\frac{d\tilde{\gamma}_n}{d\hat{\gamma}_n}(x) = \frac{1+|x|}{\int 1+|x|\,\hat{\gamma}_n(dx)}.$ Then it holds, 
\[
\sup_{n\geq 1}\tilde{\gamma}_n([-R,R]^c) \leq \frac{2M}{c(R)},\qquad \forall R\geq 1\,,
\]
and so the sequence $\tilde{\gamma}_n$ is tight. Therefore, according to the Prokhorov Theorem, extracting a subsequence if necessary, one can assume that $\tilde{\gamma}_n$ converges to some $\tilde{\gamma}$ for the weak topology. Extracting yet another subsequence if necessary, one can also assume that $\int (1+|x|)\,\gamma_n(dx)$ converges to some number $Z>0.$ The weak convergence of $\tilde{\gamma}_n$ to $\tilde{\gamma}$ means that $\int \varphi \,d\tilde{\gamma}_n \to \int \varphi \,d\gamma$ for all bounded continuous $\varphi$, which means that 
\[
\int (1+|x|)\varphi(x)\,\hat{\gamma}_n(dx) \to \int (1+|x|)\varphi(x)\,\hat{\gamma}(dx),
\]
where $\hat{\gamma}(dx) = \frac{Z}{1+|x|}\,\tilde{\gamma}(dx) \in \mathcal{P}_1(\RR).$
Invoking again \cite[Theorem 6.9]{Vil09}, this implies $\hat{\gamma}_n \to \hat{\gamma}$ as $n \to \infty$ for the $W_1$ distance.

Now we will check that $\hat{\gamma}$ is such that $\overline{\mathcal{T}_\theta}(\nu|\mu) = \mathcal{T}_\theta(\hat{\gamma},\mu)$ for all convex cost functions $\theta:\RR^+ \to \RR^+.$
First assume that $\theta$ is Lipschitz, and denote by $L_\theta$ its Lipschitz constant. According to 
\cite[Theorem 2.11]{GRST15}, the following Kantorovich duality formula holds
\[
\overline{\mathcal{T}}_\theta (\nu_n|\mu_n) = \sup_{\varphi} \left\{\int Q_\theta \varphi(x)\,\nu_n(dx) - \int \varphi(y)\,\mu_n(dy)\right\},
\]
where the supremum is taken over the set of convex functions $\varphi$ bounded from below, 
with $Q_\theta \varphi(x) := \inf_{y \in \RR} \{ \varphi(y) + \theta(|x-y|)\}$, $x \in \RR$. Define $\bar{\varphi} (y):= \sup_{x\in \RR} \{Q_\theta\varphi(x) - \theta(|x-y|)\}.$ Then it is easily checked that $\bar{\varphi} \leq \varphi$, $\bar{\varphi}$ is bounded from below and $Q_\theta \bar{\varphi} = Q_\theta\varphi$.  Moreover, being a supremum of convex and $L_\theta$-Lipschitz functions, the function $\bar{\varphi}$ is also convex and $L_\theta$-Lipschitz. Therefore, the supremum in the duality formula above can be further restricted to the class of convex functions which are $L_\theta$-Lipschitz and bounded from below. Using the fact that $W_1(\nu_n,\nu) = \sup\{\int f\,d\nu_n -\int f\,d\nu\}$ where the supremum runs over $1$-Lipschitz function and the fact that $Q_\theta \varphi$ is $L_\theta$-Lipschitz (being an infimum of such functions), we easily get the following inequality
\[
|\overline{\mathcal{T}}_\theta(\nu_n|\mu_n) - \overline{\mathcal{T}}_\theta(\nu|\mu)| \leq L_\theta W_1(\nu_n,\nu) + L_\theta W_1(\mu_n,\mu).
\]
A similar (but simpler reasoning) based on the usual Kantorovich duality for $\mathcal{T}_\theta$ yields the inequality
\[
|\mathcal{T}_\theta(\hat{\gamma}_n,\mu_n) - \mathcal{T}_\theta(\hat{\gamma},\mu)| \leq L_\theta W_1(\hat{\gamma}_n,\hat{\gamma}) + L_\theta W_1(\mu_n,\mu).
\]
Passing to the limit as $n \to \infty$ in the identity $\overline{\mathcal{T}}_\theta(\nu_n | \mu_n) = \mathcal{T}_\theta(\hat{\gamma}_n,\mu_n)$, we end up with $\overline{\mathcal{T}}_\theta(\nu | \mu) = \mathcal{T}_\theta(\hat{\gamma},\mu)$.

Now it remains to extend this identity to general convex functions $\theta$ not necessarily Lipschitz.
Let $\theta : \RR^+ \to \RR^+$ be a convex cost function (such that $\theta(0)=0$) and for all $n \geq 1$, let $\theta_n$ be the convex cost function defined by $\theta_n(x) = \theta(x)$, if $x \in [0,n]$ and $\theta_n(x) = \theta(n) + \theta'(n) (x-n)$, if $x \geq n,$ where $\theta'$ denotes the right derivative of $\theta.$ It is easily seen that $\theta_n$ is Lipschitz and that $Q_{\theta_n} \varphi$ converges to $Q_\theta \varphi$ monotonically as $n \to \infty$, for any function $\varphi$ bounded from below. 
Therefore, the monotone convergence theorem implies that for any probability measure $\gamma$, it holds $\int Q_\theta\varphi\,d\gamma = \sup_{n\geq 1} \int Q_{\theta_n}\varphi\,d\gamma$. We deduce from this that $\overline{\mathcal{T}}_\theta(\nu|\mu) = \sup_{n \geq 1} \overline{\mathcal{T}}_{\theta_n}(\nu|\mu)$ and $\mathcal{T}_\theta(\hat{\gamma}|\mu) = \sup_{n \geq 1} \mathcal{T}_{\theta_n}(\hat{\gamma},\mu)$. Since $\overline{\mathcal{T}}_{\theta_n}(\nu|\mu) =  \mathcal{T}_{\theta_n}(\hat{\gamma},\mu)$ for all $n\geq 1$, 
this ends the proof of the first part of the theorem \textit{i.e.}\ that 
$\overline{\mathcal{T}}_\theta(\nu|\mu)=\mathcal{T}_\theta(\hat{\gamma}|\mu)$).

\medskip

From the first part of the theorem we conclude that 
there exists some $\hat{\gamma} \in \mathcal{P}_1(\RR)$ such that 
$\tbar_\theta(\nu|\mu) = \mathcal{T}_\theta (\hat{\gamma},\mu)$ for the three cost functions $\theta = \theta_1, \theta_2, \theta_1 +\theta_2$. The result then follows from the well-known additivity of $\mathcal{T}_\theta$ in dimension one:
$
\mathcal{T}_{\theta_1 +\theta_2} (\hat{\gamma}, \mu) 
= 
\mathcal{T}_{\theta_1} (\hat{\gamma}, \mu) +\mathcal{T}_{\theta_2} (\hat{\gamma}, \mu)$.
This ends the proof of the theorem.
\end{proof}

\section{Dual formulation for weak transport-entropy inequalities.}\label{dualc}

In this short section we recall the Bobkov and G\"otze dual formulation of the transport-entropy inequality
 \eqref{eq:Ttheta-intro} and its extensions, borrowed from \cite{GRST15}, related to the transport-entropy inequalities of Definition \ref{defitrans}, in terms of infimum convolution inequalities. 
The results are stated in dimension one to fit our framework but hold in more general settings (see \cite{GRST15}). They will be used in the next sections.

\begin{Lem}\label{Lem:equiv-inf-conv}
Let $\mu \in \mathcal{P}_1(\RR)$ and $\theta:\RR^+ \to \RR^+$ be a convex cost function and, for all functions $g:\RR \to \RR$ bounded from below, set
\[
Q_tg(x) := \inf_{y \in \RR} \left\{f(y) + t\theta\left(\frac{|x-y|}{t}\right)\right\},\qquad  t>0, x \in \mathbb{R}.
\]  
Then the following holds. 
\begin{enumerate}
\item[(i)]  $\mu$ satisfies $\mathrm{T}(\theta)$ if and only if for all $g:\RR \to \RR$  bounded from below it holds
$$
\exp\left( \int Q_1g\,d\mu\right) \exp\left(-\int g\,d\mu\right) \leq 1 .
$$
\item[(ii)] $\mu$ satisfies $\overline{\mathrm{T}}^+(\theta)$ if and only if for all convex  $g:\RR \to \RR$ bounded from below it holds
\[
\exp\left( \int Q_1g\,d\mu\right) \int \exp(-g)\,d\mu \leq 1 .
\]
\item[(iii)] $\mu$ satisfies $\overline{\mathrm{T}}^-(\theta)$ if and only if for all convex  $g:\RR \to \RR$ bounded from below it holds
\[
 \int\exp(Q_1g)\,d\mu  \exp\left(-\int g\,d\mu\right) \leq 1 .
\]
\item[(iv)] If $\mu$ satisfies  $\overline{\mathrm{T}}(\theta)$, then for all convex $g:\RR \to \RR$ bounded from below it holds
\begin{equation}\label{eq:property tau}
\int\exp(Q_tg)\,d\mu  \int \exp (- g)\,d\mu \leq 1,
\end{equation}
with $t=2$.
Conversely, if $\mu$ satisfies \eqref{eq:property tau} for some $t>0$, then it satisfies $\overline{\mathrm{T}}(t\theta(\,\cdot\,/t)).$
\end{enumerate}
\end{Lem}
\proof
The first item is due to Bobkov and G\"otze \cite{BG99a} and is based on a combination of the well-known duality formulas for the relative entropy and for the transport cost $\mathcal{T}_\theta$. Items $(ii)$ and $(iii)$ generalize the first item to the framework of weak transport-entropy inequalities. We refer to \cite[Proposition 4.5]{GRST15} for a more general statement and for a proof (based on an extension of duality for weak transport costs). 

Finally we sketch the proof of Item $(iv)$ (which already appeared in a slightly different form in \cite[Propositions 8.2 and 8.3]{SurveyGL}). By the very definition, if $\mu$ satisfies $\overline{\mathrm{T}}(\theta)$ then it satisfies $\overline{\mathrm{T}}^\pm(\theta)$ and therefore, it satisfies the exponential inequalities given in Items $(ii)$ and $(iii)$. Note that if $g$ is convex and bounded from below then $Q_1g$ is also convex and bounded from below. Therefore it holds
\[
\exp\left( \int Q_1g\,d\mu\right) \int \exp(-g)\,d\mu \leq 1
\]
and 
\[
\int\exp(Q_1(Q_1g))\,d\mu  \exp\left(-\int Q_1g\,d\mu\right) \leq 1.
\]
Multiplying these two inequalities and noticing that $Q_1(Q_1g) = Q_{2}g$ (for a proof of this well-known semi-group property, see \textit{e.g}\ \cite[Theorem 22.46]{Vil09}) 
gives \eqref{eq:property tau} with $t=2$. The converse implication simply follows from Jensen's inequality.
\endproof

\section{A transport form of the convex Poincar\'e inequality}\label{sectionpoincare}

This section is devoted to the proof of Theorem \ref{thm:equiv-Poincare}. 
Since, from \cite{BG99b}, Item $(i)$ is equivalent to Item $(ii)$, and since it is easy to prove that Item $(iii)$ implies Item $(ii)$ we will mainly focus on the implication $(i) \Rightarrow (iii)$. Our strategy is to prove a modified logarithmic Sobolev inequality for the exponential probability measure $\tau$ and then, using a transport argument,  a modified logarithmic Sobolev inequality for general $\mu$ (satisfying the assumption of Item $(i)$) restricted to convex or concave Lipschitz functions. 
Finally, following the well-known Hamilton-Jacobi interpolation technique of \cite{BGL01}, the desired transport inequalities will follow in their dual forms (recalled in Lemma \ref{Lem:equiv-inf-conv}).

We need some notations.
Given a convex or concave function $g \colon \RR \to \RR$, we set
\begin{equation}\label{eq:nabla}
|\nabla g|(x) := \min \{ | \theta g'_-(x) + (1-\theta) g'_+(x)| ; \theta \in [0,1] \},
\end{equation}
where $g'_-$ and $g'_+$ denote the left and right derivatives of $g$ (which are well-defined everywhere).
In particular, if $g$ is convex 
\[
|\nabla g|(x) 
= \left\{\begin{array}{ll} 
|g'_+(x)| & \text{ if } g'_+(x) \leq 0 \\
0 & \text{ if } g'_-(x)\leq 0 \leq g'_+(x) \\ 
g'_-(x) & \text{ if } g'_-(x) \geq0\,,
\end{array}\right.
\]
and if $g$ is concave
$$
|\nabla g|(x) 
= \left\{\begin{array}{ll} 
|g'_-(x)| & \text{ if } g'_-(x) \leq 0 \\
0 & \text{ if } g'_+(x)\leq 0 \leq g'_-(x) \\ 
g'_+(x) & \text{ if } g'_+(x) \geq 0 .
\end{array}\right.
$$

The following result is one of the key ingredients in the proof of Theorem \ref{thm:equiv-Poincare}. 
Recall the definition of $U_\mu$ from the introduction.

\begin{Pro}\label{LSM}
Let $\mu$ be a probability measure on $\RR$. Assume that
$\sup_{x\in \RR}[U_\mu(x+1)-U_\mu(x)]\leq h$ for some $h>0$. Set $K:=2740$ and $c :=1/(10 \sqrt{2})$. Then, 
for all convex or concave and $l$-Lipschitz functions $g$ with $l\leq c/h$, it holds
\begin{equation}\label{eq:LSM}
\ent_\mu(e^g)\leq Kh^2\int_\RR |\nabla g(x)|^2e^{g(x)}\mu(dx) .
\end{equation}
\end{Pro}

The proof of Proposition \ref{LSM} is postponed to the end of this section.

\begin{proof}[Proof of Theorem \ref{thm:equiv-Poincare}]
As already mentioned above, from \cite{BG99b} we conclude that Item $(i)$ is equivalent to Item $(ii)$.
In order to make the dependency of the constants explicit in the implication $(i) \Rightarrow (ii)$, one can use a well-known expansion argument: apply \eqref{eq:LSM} to $\varepsilon f$  and take the limit $\varepsilon \to 0$, see \textit{e.g.}\ \cite{ABC+}. On the other hand, using a similar expansion argument, it is easy to prove that Item $(iii)$ implies Item $(ii)$ with $C=2D$: apply \eqref{eq:property tau} to $g=\varepsilon f$ and take the limit $\varepsilon \to 0$, see \textit{e.g.}\ \cite{SurveyGL,GRS11}. Hence, we are left with the proof of $(i)$ implies $(iii)$ which closely follows the Hamilton-Jacobi semi-group approach introduced in \cite{BGL01}. 

Let $\mu$ be a probability measure on the line and assume that Item $(i)$ of Theorem \ref{thm:equiv-Poincare} holds.
According to Proposition \ref{LSM}, for any convex or concave differentiable function 
$g$ which is $l$-Lipschitz with $l \leq c/h:=l_o$, it holds
\begin{equation}\label{eq:LSMbis}
\ent_\mu(e^g)\leq Kh^2\int_\RR |\nabla g(x)|^2e^{g(x)}\mu(dx)\,,
\end{equation}
with $K=2740$ and $c=1/(10 \sqrt{2})$.
It is easy to check that the latter is equivalent to
\begin{equation}\label{eq:LSMbis}
\ent_\mu(e^g) \leq \int \alpha^*(|\nabla g|)e^g\,d\mu,
\end{equation}
for all convex or concave $g \colon \RR \to \RR$
with 
$$
\alpha^*(v) := \sup_{u}\{uv - \alpha(u)\} = \left\{
\begin{array}{ll}
Kh^2v^2 & \text{ if } |v| \leq l_o \\ 
+\infty & \text{ if } |v|>l_o 
\end{array}
\right.
$$
the convex conjugate of $
\alpha(u) := \left\{
\begin{array}{ll}
\frac{u^2}{4Kh^2} & \text{ if } |u| \leq 2l_oKh^2 \\ 
l_o|u|-l_o^2Kh^2 & \text{ if } |u|>2l_oKh^2 
\end{array}
\right.
$.

Now, introduce the inf-convolution operators $Q_t$, for $t\in (0,1]$, defined by 
\[
Q_tf(x) := \inf_{y \in \RR}\left\{ f(y) + t\alpha\left(\frac{x-y}{t}\right)\right\},\qquad x \in \RR,\qquad  t \in (0,1],
\]
which makes sense for instance for any Lipschitz function $f$ or for any function $f$ bounded from below.
For simplicity denote by $\mathcal{F}$ the set of functions $f \colon \RR \to \RR$ that are $l$-Lipschitz and concave, $l \leq l_o$, or convex and bounded below.
Then, $Q_t$ satisfies the following technical properties:
\begin{itemize}
\item[$(a)$] If $f$ is convex, then $Q_tf$ is convex.
\item[$(b)$] If $f$ is concave and Lipschitz, then $Q_tf$ is concave.
\item[$(c)$] If $f \in \mathcal{F}$,  then  $Q_tf$ is $l_o$-Lipschitz.
\item[$(d)$] If $f\in \mathcal{F}$, then the function $u(t,x) := Q_t f(x)$ satisfies the following Hamilton-Jacobi equation
\begin{equation}\label{eq:HJ}
\frac{d}{dt_+}u(t,x) + \alpha^*(|\nabla^- u|)(t,x) =0, \qquad \forall t \in (0,1],  \qquad \forall x \in \RR,
\end{equation}
where $\frac{d}{dt_+}$ is the right time-derivative and 
$|\nabla^- u(t,x)| = \limsup_{y \to x} \frac{[u(t,y)-u(t,x)]_-}{|y-x|}$ (where as usual $[X]_-:=\max(-X,0)$ denotes the negative part).
\end{itemize}
Item $(a)$ is easy to check and is a general fact about infimum convolution of two convex functions ($f$ and $\alpha$). Item $(b)$ follows from the fact that, after change of variables, $Q_tf(x) = \inf_{u}\left\{f(x-u) + t\alpha\left(\frac{u}{t}\right)\right\}$ so that $Q_tf$ is an infimum of concave functions and is therefore also concave.
As for Item $(c)$ we observe that $x \mapsto t\alpha\left(\frac{x-y}{t}\right)$ is $l_o$-Lipschitz for any $y$
so that $Q_tf$ is also  $l_o$-Lipschitz as an infimum of $l_o$-Lipschitz functions.
A proof of Item $(d)$ can be found in \cite{GRS14} or \cite{AGS14}. We observe that the conclusions of Item $(c)-(d)$ hold in much more general settings.

With these properties and definitions in hand, let 
 $f \in \mathcal{F}$ and 
 (following \cite{BGL01}) define
\[
F(t):=\frac{1}{t}\log\left(\int_\RR e^{tQ_tf}\,d\mu\right), \qquad t \in (0,1].
\]
The function $F$ is right differentiable at every point $t>0$ (thanks to the above technical properties of $Q_t$, see \textit{e.g.}\ \cite{GRS14} for details)  and it holds
\begin{align*}
\frac{d}{dt_+}F(t)
&=
\frac{1}{t^2} \frac{1}{\int_\RR e^{tQ_tf}\,d\mu}
\left( \ent_\mu\left(e^{tQ_tf}\right)
+t^2 \int_\RR  \left(\frac{d}{d t_+}Q_tf\right) e^{tQ_tf}\, d\mu
\right)\\
&= 
\frac{1}{t^2} \frac{1}{\int_\RR e^{tQ_tf}\,d\mu}
\left( \ent_\mu\left(e^{tQ_tf}\right)
-Kh^2t^2 \int_\RR  |\nabla^- Q_tf|^2e^{tQ_tf}\, d\mu\right)\\
& \leq  
\frac{Kh^2}{\int_\RR e^{tQ_tf}\,d\mu}
\left(  \int |\nabla Q_tf|^2 e^{tQ_tf}\,d\mu
- \int_\RR |\nabla^- Q_tf|^2 e^{tQ_tf}\, d\mu\right)  \leq 0,
\end{align*}
where the second equality follows from \eqref{eq:HJ}, the first inequality from \eqref{eq:LSMbis} applied to the function $g=tQ_tf$ (which is convex or concave and $tl_o$-Lipschitz) and the last inequality from the fact that for a convex or concave function $g$, $|\nabla g| \leq |\nabla^- g|$ (we recall that $|\nabla g|$ is  defined in \eqref{eq:nabla}).

Thus the function $F$ is non-increasing and satisfies $F(1) \leq \lim_{t \to 0} F(t) = \int fd\mu$. In other words,
\begin{equation}\label{eq:dual-conv-conc}
\int e^{Q_1f}\,d\mu \leq e^{\int f\,d\mu} \qquad \forall f \in \mathcal{F} .
\end{equation}
Now according to Item $(iii)$ of Lemma \ref{Lem:equiv-inf-conv} 
 one concludes (on the one hand) that $\mu$ satisfies the transport-entropy inequality $\overline{\mathrm{T}}^-(\alpha)$:
$\overline{\mathcal{T}}_\alpha (\mu|\nu) \leq H(\nu|\mu)$, for all $\nu \in \mathcal{P}_1(\RR)$.

On the other hand, applying \eqref{eq:dual-conv-conc} to $f = -Q_1g$ with $g$ convex and bounded from below (so that $f$ is concave and $l_o$-Lipschitz) yields to
$e^{\int Q_1g\,d\mu} \int e^{Q_1 (-Q_1 g)}\,d\mu \leq 1$. 
Since $Q_1(-Q_1 g) \geq -g$ we end up with
\[
e^{\int Q_1g\,d\mu} \int e^{-g}\,d\mu \leq 1,
\]
for all $g$ convex and bounded from below. According to Item $(ii)$ of Lemma \ref{Lem:equiv-inf-conv}, this implies that $\mu$ satisfies the transport-entropy inequality $\overline{\mathrm{T}}^+(\alpha)$:
$
\overline{\mathcal{T}}_\alpha (\nu|\mu) \leq H(\nu|\mu)$, for all $\nu \in \mathcal{P}_1(\RR)$,
which completes the proof.
\end{proof}

The end of the section is dedicated to the proof of Proposition \ref{LSM}.

\begin{proof}[Proof fo Proposition \ref{LSM}]
Let $K$ and $c$ be defined by Lemma \ref{lem:expo} below.
We may deal first with convex functions $g$ and divide the proof into three different (sub-)cases: $g$ monotone (non-decreasing and then non-increasing), and $g$ arbitrary.
 
Assume first that $g$ is convex \emph{non-decreasing} and $l$-Lipschitz with $l\leq c/h$. Set $f=g\circ U_\mu$ (recall that $U_\mu$ is defined in the introduction). Then, since $g$ is non-decreasing, and since $U_\mu(x-1) \leq U_\mu(x)-h$ by assumption, for all $x\in \RR$, it holds
\[
f(x)-f(x-1)\leq g(U_\mu(x))-g(U_\mu(x)-h)\leq lh\leq c, \qquad \forall x \in \RR .
\]
Therefore, since $\mu$ is the image of $\tau$ under the map $U_\mu$, 
Lemma \ref{lem:expo} (apply \eqref{expo:x-1} to $f$) and the latter guarantee that 
\begin{align*}
\ent_\mu(e^g)
& = 
\ent_\tau(e^f) 
\leq 
K\int_\RR \left(f(x)-f(x-1)\right)^2e^{f(x)}\tau(dx) \\
& \leq
K\int_\RR \left(g(U_\mu(x))-g(U_\mu(x)-h)\right)^2e^{g(U_\mu(x))}\tau(dx) \\
& =
K \int_\RR \left(g(x)-g(x-h)\right)^2e^{g(x)}\mu(dx) 
 \leq 
Kh^2\int_\RR  |\nabla g(x)|^2e^{g(x)}\mu(dx),
\end{align*}
where the last inequality is due to the fact that $g$ is convex and non-decreasing 
and therefore satisfies $0\leq g(x)-g(x-h) \leq g_-'(x)h = |\nabla g(x)| h$.
As a conclusion we proved \eqref{eq:LSM} for all convex \emph{non-decreasing} and $l$-Lipschitz functions $g$ with $l\leq c/h$.

Now suppose that $g$ is convex, \emph{non-increasing} and $l$-Lipschitz with $l\leq c/h$ and set $f(x)= g(U_\mu(-x))$. The function $f$ is non-decreasing and, since $U_\mu(-x+1) \geq U_\mu(-x)+h$ by assumption, satisfies
\[
f(x)- f(x-1) = g(U_\mu(-x)) - g(U_\mu(-x+1)) \leq g(U_\mu(-x)) - g(U_\mu(-x)+h) \leq c.
\]
Similarly to the previous lines, Lemma \ref{lem:expo} implies that
\begin{align*}
\ent_\mu(e^g) 
=
\ent_\tau(e^f) 
& \leq 
K\int_\RR \left(g(U_\mu(-x))-g(U_\mu(-x+1))\right)^2e^{g(U_\mu(-x))}\tau(dx) \\
& \leq 
K \int_\RR \left(g(U_\mu(-x))-g(U_\mu(-x)+h)\right)^2e^{g(U_\mu(-x))}\tau(dx) \\
& = K\int_\RR \left(g(x)-g(x+h)\right)^2e^{g(x)}\mu(dx)
\leq 
Kh^2\int_\RR |\nabla g (x)|^2e^{g(x)}\mu(dx),
\end{align*}
where we used the symmetry of $\tau$ and that $0\leq g(x)-g(x+h) \leq g_+'(x)(-h)=|\nabla g(x)| h$.
Therefore we proved \eqref{eq:LSM} for all convex \emph{non-increasing} and $l$-Lipschitz functions $g$ with $l\leq c/h$.

Finally, consider an arbitrary convex and $l$-Lipschitz function $g$ with $l\leq c/h$ and
assume without loss of generality that $g$ is not monotone. Being convex, there exists some $a \in \RR$ such that $g$ restricted to $(-\infty,a]$ is non-increasing and $g$ restricted to $[a,\infty)$ is non-decreasing. Subtracting $g(a)$ if necessary, one can further assume that $g(a)=0$ since \eqref{eq:LSM} is invariant by the change of function $g \to g+C$ (for any constant $C$). Set $g_1=g\mathbf{1}_{(-\infty,a]}$ and $g_2 = g\mathbf{1}_{(a,\infty)}$. The functions $g_1$ and $g_2$ are convex, monotone and $l$-Lipschitz. Therefore, according to the two previous sub-cases, it holds
\[
\ent_\mu(e^{g_1}) \leq  Kh^2 \int_{-\infty}^a |\nabla g(x)|^2 e^{g(x)}\mu(dx) \mbox{ and } 
\ent_\mu(e^{g_2})\leq  Kh^2 \int_a^{+\infty} |\nabla g(x)|^2 e^{g(x)}\mu(dx) .
\]
So what remains to prove is the following sub-additivity property of the entropy functional
\[
\ent_\mu(e^{g_1+g_2}) \leq \ent_\mu(e^{g_1}) + \ent_\mu(e^{g_2}),
\]
which, since $\int ge^gd\mu = \int g_1e^{g_1}d\mu+\int g_2e^{g_2}d\mu$, amounts to proving that
\[
\int e^{g_1}d\mu \log\left(\int e^{g_1}d\mu\right) + \int e^{g_2}d\mu \log\left(\int e^{g_2}d\mu\right)
\leq 
\int e^g\,d\mu \log\left(\int e^gd\mu\right) .
\]
Setting $A= \int e^{g_1}d\mu-1$, $B= \int e^{g_2}d\mu-1$ and $X=\int e^{g}d\mu$ and observing that
$A+B+1=X$ the latter is equivalent to proving that 
$$
(A+1) \log(A+1) + (B+1) \log(B+1) \leq X \log X\,,
$$
which follows from the sub-additivity property of the convex function $\Phi \colon x \mapsto (x+1) \log(x+1)$ on $[0,\infty)$, that satisfies $\Phi(0)=0$.
This completes the proof when $g$ is convex.

The case $g$ concave follows the same lines (use \eqref{expo:x+1} instead of  \eqref{expo:x-1}). Details are left to the reader.
\end{proof}

In the proof of Proposition \ref{LSM} we used the following lemma which is a (discrete) variant of
a result by Bobkov and Ledoux \cite{BL97} and an entropic counterpart of a result (involving the variance) by Bobkov and G\"otze (see \cite[Lemma 4.8]{BG99b}).

\begin{Lem}\label{lem:expo}
For all non-decreasing function $f \colon \RR \to \RR$ with $f(x)-f(x-1)\leq 1/(10\sqrt{2})$,  $x \in \RR$, it holds
\begin{equation}\label{expo:x-1}
\ent_\tau(e^f)\leq 2740 \int_\RR \left(f(x)-f(x-1)\right)^2e^f d\tau.
\end{equation}
and
\begin{equation}\label{expo:x+1}
\ent_\tau(e^f)\leq 2740 \int_\RR \left(f(x+1)-f(x)\right)^2e^f d\tau.
\end{equation}
\end{Lem}

\begin{proof}
Let $\tau^+$ be the exponential probability measure on $\RR^+$:  $\tau^+(dx) = e^{-x} \mathds{1}_{[0,\infty)}dx$.
We shall use the following fact, borrowed from \cite[Lemma 4.7]{BG99b} (with $a=0$ and $h=1$ so that the constant $c(a,h)$ appearing in \cite{BG99b} can be explicitly bounded by $1/200$):
for all $f \colon [-1,\infty) \to \RR$ non-decreasing and satisfying $f(0)=0$, it holds
\begin{equation}\label{pluie}
\int f^2 d\tau^+ \leq 200 \int (f(x)-f(x-1))^2 d\tau^+(x) .
\end{equation}
We will first prove \eqref{expo:x-1}.
Since  \eqref{expo:x-1} is invariant by the change of function $f \to f+C$ for any constant $C$, we may assume without loss of generality that $f(0)=0$. Set $\tilde f(y):=-f(-y)$, $y\in \RR$ and observe that $f$ is
non-decreasing. 
Since
$u\log u\geq u-1$ for all $u\geq 0$, one has 
\begin{align} \label{200}
\ent_\tau(e^f)
&\leq 
\int (fe^f-e^f+1) \,d\tau=\int \left(\int_0^1tf^2e^{tf}dt\right)d\tau \nonumber \\
& = 
\frac{1}{2}\int_0^{\infty} f^2 \left(\int_0^1 te^{tf}dt\right)d\tau^+ + \frac{1}{2}\int_0^{\infty} \tilde{f}^2 \left(\int_0^1 te^{-t\tilde{f}}dt\right)d\tau^+ \nonumber \\
&\leq 
\frac14 \int f^2 e^f d\tau^++\frac14 \int \tilde f^2 d\tau^+,
\end{align} 
where the last inequality comes from the fact both $f$ and $\tilde{f}$ are non-negative on $\RR^+$.
Now suppose that the function $f$ is such that $f(y)-f(y-1)\leq c$ for all $y\in \RR$ and some $c \in (0,1)$.
Our aim is to bound each term in the right hand side of the latter.
%
 
By \eqref{pluie} applied to the function $\tilde f$, one has 
\begin{align*}
 \int \tilde f^2 & d\tau^+  
 \leq 
 200  \int (\tilde f(y)-\tilde f(y-1))^2d\tau^+(y)  = \frac{200}{\int e^{-\tilde f} d\nu}
 \int (\tilde f(y)-\tilde f(y-1))^2 e^{-\tilde f(y)}\,d\tau^+(y) \\
 &\leq
 200 \exp\left(\frac{\int \tilde f(y)(\tilde f(y)-\tilde f(y-1))^2\,d\tau^+(y)}{ \int (\tilde f(y)-\tilde f(y-1))^2\,d\tau^+(y)}\right) \int (\tilde f(y)-\tilde f(y-1))^2 e^{-\tilde f(y)}\,d\tau^+(y).
 \end{align*}
where we set $\frac{d\nu}{d\tau^+}(y)=\frac {(\tilde f(y)-\tilde f(y-1))^2}{ \int (\tilde f(y)-\tilde f(y-1))^2 \,d\tau^+(y)}$ and we used Jensen's inequality to guarantee that $1/\int e^{-\tilde f}d\nu \leq e^{\int \tilde fd\nu}$.
By Cauchy-Schwarz' inequality and using  \eqref{pluie} again, we get
\begin{align*}
\int \tilde f(y)(\tilde f(y)-\tilde f(y-1))^2d\tau^+(y)
& \leq 
\left(\int (\tilde f(y)-\tilde f(y-1))^4d\tau^+(y)\right)^{1/2} \left(\int \tilde f^2d\tau^+\right)^{1/2}\\
 &\leq 
 \sqrt{200} c \int (\tilde f(y)-\tilde f(y-1))^2d\tau^+(y).
 \end{align*}
 It finally follows that 
 \begin{align*}
\int \tilde f^2  d\tau^+
&\leq  
200 e^{\sqrt{200} c} \int (\tilde f(y)-\tilde f(y-1))^2 e^{-\tilde f(y)}\,d\tau^+(y)\\
 &= 
 200 e^{\sqrt{200} c} \int_{-\infty}^0 ( f(y+1)- f(y))^2 e^{ f(y)} e^y\,dy\\
&= 
200 e^{\sqrt{200} c-1} \int_{-\infty}^1 ( f(y)- f(y-1))^2 e^{ f(y-1)} e^y\,dy \\
&\leq 400 e^{\sqrt{200} c+1} \int_{-\infty}^1 ( f(y)- f(y-1))^2 e^{ f(y)} \,d\tau(y)\,,
\end{align*}
where in the last line we used that $e^y/(e^{-|y|}/2) \leq 2e^2$ for all $y \leq 1$.

Next we deal with the first term in the right hand side of \eqref{200}.
Our aim is to apply \eqref{pluie} to $g=fe^{f/2}$. 
Observe that, since $f$ is non-decreasing, $f(x) \geq f(-1) \geq -c + f(0)=-c \geq -1$ so that, since $x \mapsto 
xe^{x/2}$ is non-increasing on $[-2,\infty)$ we are guaranteed that $g$ is non-decreasing on $[-1,\infty)$ and therefore that we can apply \eqref{pluie} to $g$.
Applying \eqref{pluie} to $g=fe^{f/2}$ and using the inequality 
 \[
 0\leq be^{b/2} - ae^{a/2} \leq (b-a)e^{b/2} + \frac{b}{2}(b-a)e^{b/2}, \quad -2\leq a \leq b,
 \]
 we get
 \begin{align*}
 & B := \int f^2 e^f d\tau^+ 
 \leq  
 200\int \left(f(y)e^{f(y)/2}-f(y-1)e^{f(y-1)}\right)^2\,d\tau^+(y)\\
 &\leq  
 400 \int (f(y)-f(y-1))^2 e^{f(y)}\,d\tau^+(y)
 +100 \int f^2(y)(f(y)-f(y-1))^2 e^{f(y)}\,d\tau^+(y)\\
 &\leq  
 400 \int (f(y)-f(y-1))^2 e^{f(y)}\,d\tau^+(y) + 100 c^2 B.
 \end{align*}
Therefore, provided $c < 1/10$ we end up with $B \leq 400/(1-100c^2) \int (f(y)-f(y-1))^2 e^{f(y)}\,d\tau^+(y)$.
Hence, plugging the previous two bounds into \eqref{200} and choosing $c=1/\sqrt{200}$, Inequality \eqref{expo:x-1} follows with the better constant $939$ in factor of the right hand side.

 To obtain \eqref{expo:x+1} from \eqref{expo:x-1}, it suffices to observe that, by a simple change of variables
 \begin{align*}
  \int ( f(y)- f(y-1))^2 e^{ f(y)} \, d\tau(y)&= \int ( f(x+1)- f(x))^2 e^{ f(x+1)}  \frac{e^{-|y+1|}}2 \,dy\\
  &\leq e^{c+1}  \int ( f(x+1)- f(x))^2 e^{ f(x)}\,d\tau(x)
  \end{align*}
  and that $939 e^{c+1} \leq 2740$ for $c=1/\sqrt{200}$. This ends the proof.
\end{proof}

\section{Proof of Theorem \ref{thefinal}}\label{sectionfinal}

In this final section we prove  Theorem \ref{thefinal}. As mentioned in the introduction, we may need to decompose the cost $\theta$ into the sum of two costs, one that takes care of the behavior near $0$ (the cost $\theta_1$) and the other one vanishing in a neighborhood of $0$ (the cost $\theta_2$). The next theorem deals with the latter.

\begin{The}\label{thebeta}
Let $\mu\in \mathcal{P}_1(\RR)$ and $\beta:\RR^+ \to \RR^+$ be a convex cost function such that $\{t \in \RR^+ : \beta(t) =0\} = [0,t_o]$, where $t_o>0$ is some positive constant. The following are equivalent:
\begin{enumerate}
\item[$(i)$] There exists $a>0$ such that $\mu$ satisfies the transport-entropy inequality $\mathrm{T}(\beta(a\,\cdot\,))$.
\item[$(ii)$] There exists $a'>0$ such that $\mu$ satisfies the weak transport-entropy inequality $\overline{\mathrm{T}}(\beta(a'\,\cdot\,))$.
\item[$(iii)$] There exists $b>0$ such that $\max(K^+(b),K^-(b)) < \infty $, where 
\[
K^+(b):=\sup_{x\geq m}\frac{1}{\mu(x,\infty)}\int^{\infty}_{x}e^{\beta(b(u-x))}\mu(du),
\] 
\[
K^-(b):=\sup_{x\leq m}\frac{1}{\mu(-\infty,x)}\int_{-\infty}^{x}e^{\beta(b(x-u))}\mu(du),
\]
and $m$ is a median of $\mu$. (Here we use the convention $0/0=0$.)
\item[$(iv)$] There exists $d>0$ such that
\[
|U_\mu(u)-U_\mu(v)|\leq \frac{1}{d}\beta^{-1}(|u-v|),\qquad  \forall u\neq v\in \RR.
\]
(Note that $\beta^{-1}$ is well defined on $(0,\infty)$.)
\end{enumerate}
In particular, 
\begin{itemize}
\item $(i) \Rightarrow (ii)$ with $a'=a$,
\item $(ii) \Rightarrow (iii)$ with $b=a'/2$,
\item $(ii) \Rightarrow (iv)$ with  $d = a' \frac{t_o}{ 8\beta^{-1}(\log 3)}$, 
\item $(iv) \Rightarrow (ii)$ with $a' = d \frac{t_o}{9 \beta^{-1}(2)}.$
\end{itemize}

\end{The}

\begin{proof}[Proof of Theorem \ref{thebeta}]
The equivalence between Items $(i)$, $(iii)$ and $(iv)$ is proved in  \cite[Theorem 2.2]{Goz12}, with some explicit dependency between the constants.
In order to complete the proof of Theorem \ref{thebeta} we need to show that $(i) \Rightarrow (ii) \Rightarrow (iii)$.

It follows from Jensen's inequality that
$
\mathcal{T}_{\beta(a\,\cdot\,)}(\mu,\nu) \geq \max \left(\mathcal{T}_{\beta(a\,\cdot\,)}(\nu|\mu) ; \mathcal{T}_{\beta(a\,\cdot\,)}(\mu|\nu)\right).
$
Therefore $(i)$ implies $(ii)$ with $a'=a$.

Next we will show that $(ii)$ implies $(iii)$. Assume that $\mu$ satisfies $\overline{\mathrm{T}}(\beta(a'\,\cdot\,))$ for some $a'>0$.  According to Item $(iv)$ of Lemma \ref{Lem:equiv-inf-conv}, for all convex functions $g:\RR \to \RR$ bounded from below, it holds
\[
\int \exp(Qf)\,d\mu \int e^{-f}\,d\mu \leq 1, \qquad
\mbox{where} \quad
Qf(x)=\inf_{y\in \RR}\{f(y)+2\beta(a'|y-x|/2)\}.
\]
Consider the convex function $f_x$ which equals to $0$ on $(-\infty,x]$ and $\infty$ otherwise. Then $Qf_x(y)=0$ on $(-\infty,x]$ and $Qf_x(y)=2\beta(a'(y-x)/2)$ on $(x,\infty)$. Thus, applying the inequality above to $f_x$  yields
\[
\left(\mu(-\infty,x]+\int_{(x,\infty)}e^{2\beta(a'(y-x)/2)}\mu(dy)\right)\mu(-\infty,x]\leq 1.
\]
Considering $x \geq m$ yields that $K^+(a'/2)\leq 3$. One proves similarly that $K^-(a'/2)\leq 3$. This shows that $(ii)$ implies $(iii)$ with $b=a'/2$. This achieves the proof of Theorem \ref{thebeta}.
\end{proof}

We are now in a position to prove Theorem \ref{thefinal}.
\proof[Proof of Theorem \ref{thefinal}]
Let $\theta:\RR^+ \to \RR^+$ be a convex cost function such that $\theta(t) = t^2$ on $[0,t_o]$ for some $t_o>0.$
Let us define $\theta_1(t)=t^2$ on $[0,t_o]$ and $\theta_1(t) = 2tt_o - t_o^2$ on $[t_o,+\infty)$ and $\theta_2(t)=[\theta(t)-t^2]_+.$ Note that $\theta_1$ and $\theta_2$ are both convex and that $\theta_2$ vanishes on $[0,t_o]$ and that $\max(\theta_1,\theta_2)\leq \theta \leq \theta_1 + \theta_2.$

First assume that $\mu$ satisfies the weak transport-entropy inequality $\overline{\mathrm{T}}(\theta (a\,\cdot\,))$ for some $a>0$ (\textit{i.e.}\ Item $(i)$ of Theorem \ref{thefinal}). Then, since $\theta \geq \theta_2$ it clearly satisfies $\overline{\mathrm{T}}(\theta_2(a\,\cdot\,)).$ According to Theorem \ref{thebeta}, the mapping $U_\mu$ sending the exponential measure on $\mu$ satisfies the condition:
\begin{equation}\label{eq:contraction-rappel}
\sup_{x \in \RR}U_\mu(x+u) - U_\mu(x) \leq \frac{1}{b} \theta_2^{-1}(u),\qquad \forall u>0,
\end{equation}
with $b = a \kappa_1$, where $\kappa_1 = t_o/ (8\theta_2^{-1}(\log 3))$. Since $\theta_2^{-1}(u) = \theta^{-1}(u+t_o^2)$ this proves Item $(ii)$ of Theorem \ref{thefinal} with the announced dependency between the constants.

Now assume that $\mu$ satisfies Item $(ii)$ of Theorem \ref{thefinal}, or equivalently \eqref{eq:contraction-rappel} for some $b>0$. 
Recall that we set, in Theorem \ref{thm:equiv-Poincare}, $\kappa:=5480$ and $c:=1/(10 \sqrt{2})$.
Then, observe that, plugging $u=1$ into \eqref{eq:contraction-rappel} and using Theorem \ref{thm:equiv-Poincare}, one concludes that $\mu$ satisfies $\overline{\mathrm{T}} (\alpha)$ with $\alpha$ defined by $\alpha(u) = \bar{\alpha} (u / \sqrt{2D})$, with $D = \kappa \left(\theta^{-1}(1+t_o^2)\right)^2 \frac{1}{b^2}$ and 
\[
\bar{\alpha} (v) = \left\{
\begin{array}{ll}
v^2 & \text{if } |v| \leq c\sqrt{\kappa/2} \\
c\sqrt{\kappa} |v| - \frac{c^2 \kappa}{2} & \text{if } |v| >c\sqrt{\kappa/2} 
\end{array}\right.
=
\left\{
\begin{array}{ll}
v^2 & \text{if } |v| \leq \sqrt{137/10} \\
2\sqrt{137} |v| - \frac{137}{5} & \text{if } |v| >c\sqrt{137/10} .
\end{array}\right.
\]
It is not difficult to check that $\bar{\alpha}$ compares to $\theta_1$. More precisely, for all $v \in \RR$,
it holds
\[
\bar{\alpha} (v) 
\geq 
\theta_1\left(\max\left(\frac{c\sqrt{\kappa/2}}{t_o}; 1\right) |v| \right)
= 
\theta_1\left(\max\left(\frac{\sqrt{137/10}}{t_o}; 1\right) |v| \right) .
\]  
Therefore $\mu$ satisfies $\overline{\mathrm{T}}(\theta_1(a_1'\,\cdot\,))$, and by monotonicity 
$\overline{\mathrm{T}}(\theta_1(a_1\,\cdot\,))$
with 
\[
a_1' 
:= 
\frac{\max((c\sqrt{\kappa/2})/t_o; 1)}{\sqrt{2\kappa}\theta^{-1}(1+t_o^2)} b
= 
\frac{\max\left(\frac{\sqrt{137/10}}{t_o}; 1\right)}{4\sqrt{685}\theta^{-1}(1+t_o^2)} b
\geq \frac{1}{105} \frac{\max(1,t_o)}{t_o \theta^{-1}(1+t_o^2)} b =: a_1 
.
\]
On the other hand, according to Theorem \ref{thebeta}, $\mu$ also satisfies $\overline{\mathrm{T}}(\theta_2(a_2\,\cdot\,))$, with $a_2 = \frac{t_o}{\theta^{-1}(2+t_o^2)} b.$ Letting $a = \min(a_1,a_2)$, one concludes that $\mu$ satisfies both $\overline{\mathrm{T}}(\theta_1(a\,\cdot\,))$ and $\overline{\mathrm{T}}(\theta_2(a\,\cdot\,))$. Hence, since $\theta(at) \leq \theta_1(at)+\theta_2(at)$ and according to \eqref{eval}, it holds 
\begin{align*}
\mathcal{T}_{\theta(a\,\cdot\,)} (\nu |\mu)& \leq \mathcal{T}_{\theta_1(a\,\cdot\,)+\theta_2(a\,\cdot\,)} (\nu |\mu) =  \mathcal{T}_{\theta_1(a\,\cdot\,)} (\nu |\mu) + \mathcal{T}_{\theta_2(a\,\cdot\,)} (\nu |\mu)\\
& \leq 2 H(\nu|\mu),
\end{align*}
and so $\mu$ satisfies $\overline{\mathrm{T}}^+(\frac{1}{2}\theta(a\,\cdot\,))$. By convexity of $\theta$ and since $\theta(0)=0$, it holds $\frac{1}{2}\theta(2at) \geq \theta(at)$, and so $\mu$ satisfies  $\overline{\mathrm{T}}^+(\theta((a/2)\,\cdot\,))$.
Finally we observe that
$$
\frac{a}{2} = \frac{b}{210} \min \left(\frac{\max(1,t_o)}{t_o \theta^{-1}(1+t_o^2)}; \frac{105 t_o}{\theta^{-1}(2+t_o^2)} \right)
\geq \frac{b}{210} \frac{\min(1,t_o)}{\theta^{-1}(2+t_o^2)} =: \kappa_2 b
$$
so that, by monotonicity, $\mu$ satisfies  $\overline{\mathrm{T}}^+(\theta(\kappa_2b\,\cdot\,))$.
 The same reasoning yields the conclusion that $\mu$ satisfies $\overline{\mathrm{T}}^-(\kappa_2b\,\cdot\,))$, which completes the proof.
\endproof

\bibliographystyle{amsplain}

\providecommand{\MR}{\relax\ifhmode\unskip\space\fi MR }
\providecommand{\MRhref}[2]{%
  \href{http://www.ams.org/mathscinet-getitem?mr=#1}{#2}
}
\providecommand{\href}[2]{#2}

\end{document}